\documentclass[12pt,a4paper]{article}
\usepackage{fullpage}
\usepackage{authblk}

\usepackage{mathrsfs}
\usepackage{amsfonts}
\usepackage{amssymb}
\usepackage{bm}
\usepackage{enumerate}
\usepackage{tikz}
\usepackage{graphicx}
\usepackage{stfloats}
\usepackage{amsmath}

\usepackage{mathrsfs}
\usepackage{amsfonts}
\usepackage{amssymb}
\usepackage{bm}
\usepackage{enumerate}
\usepackage{tikz}
\usepackage{graphicx}
\usepackage{stfloats}
\usepackage{amsmath}
\usepackage[colorlinks,
            linkcolor=red,
            anchorcolor=blue,
            citecolor=green
            ]{hyperref}

\usepackage{amsfonts}
\usepackage{moreverb}
\usepackage{color}
\usepackage{bm}
\usepackage[english]{babel}
\usepackage{amsmath,cases,amsfonts,amssymb,mathrsfs}
\usepackage{graphicx, multirow, subfigure, txfonts}

\newtheorem{defn}{Definition}[section]
\newtheorem{remark}{\bf Remark}[section]
\newtheorem{lemma}{\bf Lemma}[section]
\newtheorem{example}{\bf  Example}[section]
\newtheorem{theorem}{\bf  Theorem}[section]
\newenvironment{proof}{\noindent{\em Proof:}}{\quad \hfill$\Box$\vspace{2ex}}
\newtheorem{coro}{\bf Corollary}[section]

\def\i{{\bf i}}

 \def\j{{\bf j}}
 \def\k{{\bf k}}

\usepackage{amssymb}
 \usepackage{amsfonts}

\numberwithin{equation}{section}

\title{Quaternion Fourier and Linear Canonical Inversion Theorems}

\author{Xiao Xiao Hu\thanks{huxiaoxiao3650@163.com}  }

\author{Kit Ian Kou\thanks{Corresponding author: kikou@umac.mo}}

\affil{\normalsize{Department of Mathematics, Faculty of Science and Technology, University of Macau, Macao, China}}
\date{}

\begin{document}

   \maketitle

\begin{abstract}
\normalsize

The Quaternion Fourier transform (QFT) is one of the key tools in studying color image processing. Indeed, a deep understanding of the QFT has created the color images to be transformed as whole, rather than as color separated component. In addition, understanding the QFT paves the way for understanding other integral transform, such as the Quaternion Fractional Fourier transform (QFRFT), Quaternion linear canonical transform (QLCT) and Quaternion Wigner-Ville distribution. The aim of this paper is twofold: first to provide some of the theoretical background regarding the Quaternion bound variation function. We then apply it to derive the Quaternion Fourier and linear canonical inversion formulas. Secondly, to provide some in tuition for how the Quaternion Fourier and linear canonical inversion theorems work on the absolutely integrable function space.
\end{abstract}

\begin{keywords}
Inversion theorem;  functions of bound variation; quaternion
Fourier transform.
\end{keywords}
\section{Introduction}\label{S1}

Fourier inversion plays a crucial role in signal processing and communications: it tells us how to convert an continuous signal into some trigonometric functions, which can then be processed digitally or coded on a computer \cite{papoulis1960fourier}.
 The inversion theorem (also name Fourier integral theorem) states that the input signal or image can be retrieved from its according frequency function via the inversion Fourier transform. In mathematics \cite{stein1971introduction,papoulis1960fourier,PWJ2000}, there are two common conditions versions of the Fourier inversion theorem hold for integrable function.

\begin{itemize}
  \item For a real-valued integrable function $f$, if $f$ is a function of bounded variation in the neighborhood of $x_{0},$ then
  \begin{eqnarray*}
 \frac{f(x_{0}+0)+f(x_{0}-0)}{2}= \lim_{\begin{subarray}
  MM \to \infty
\end{subarray}}\frac{1}{2\pi}\int_{-M}^{M}\widehat{f}(u)e^{\i ux_{0}}du .
\end{eqnarray*}

  \item
  If both $f$ and $\widehat{f}$ are integrable, then $f$ can be recovered from its Fourier transform $\widehat{f},$
  \begin{eqnarray*}
  f(x)=\frac{1}{2\pi}\int_{\mathbb{R}}\widehat{f}(u)e^{\i ux}du,
\end{eqnarray*}
for almost every $x$, where $\widehat{f}(u)=\int_{\mathbb{R}}f(x)e^{-\i ux}du.$
\end{itemize}

There have been numerous proposals in the literature to generalize the classical Fourier
transform (FT) by making use of the Hamiltonian quaternion algebra \cite{ouyang2015color, yang2015novel}, namely quaternion Fourier transforms (QFTs).
Quaternion algebra \cite{hamilton1866elements} is thought to generalize the classical theory of holomorphic functions of one complex variable onto the multidimensional situation, and to provide the foundations for a refinement of classical harmonic analysis. In the meantime, quaternion algebra has become a well established mathematical discipline and an active area of research with numerous connections to other areas of both pure and
applied mathematics. In particular, there is a well developed theory of quaternion
analysis with many applications to Fourier analysis and partial differential
equations theory, as well as to other fields of physics and engineering \cite{took2009quaternion,took2011augmented,gou2015three}.
The QFTs play a vital role in the representation of multidimensional
(or quaternionic) signals. They transform a 2D real (or quaternionic-valued) signals into
the quaternionic-valued frequency domain signals. The four components of the QFTs separate four
cases of symmetry into real signals instead of only two as in the complex Fourier transforms \cite{yang2015novel}.
The QFTs have been found many applications in color image processing, especially in color-sensitive smoothing, edge detection and data compression etc \cite{evans2000hypercomplex, evans2000colour,pei1999color,sangwine1996fourier,ell2007hypercomplex,ell2014quaternion}.  In \cite{chen2015pitt}, the authors studied the inversion theorem of QFTs for square integrable functions, the convergence of the quaternion Fourier integral is in mean square norm.
To the best of our knowledge, there has been no previous work (at least systematically) studying the conditions of quaternion Fourier inversion theorem for integrable functions. Therefore, it is worthwhile and interesting to investigate them.

Over the last few years, there has been a growing interest of the classical linear canonical transform (LCT) in engineering, computer sciences, physics and applied mathematics \cite{kou2012windowed, healy2016linear}. The LCT is a linear integral transformation of a four-parameter function and it can be considered as the generalization of the fractional Fourier transform (FRFT) and the FT. Comparing to the FRFT and the FT, the LCT has shown to be more flexible for signal processing. Therefore, it is desirable to extend the LCT to higher dimensions and to study its properties. To this end, quaternionic analysis offers possibilities of generalizing the underlying function theory from 2D to 4D, with the advantage of meeting exactly the same goals. See Refs. \cite{ell2014quaternion, hitzer2007quaternion} for a more complete account of this subject and
related topics. A higher-dimensional extension of the LCT within the Clifford analysis setting was first studied in \cite{kou2013generalized}. The paper generalizes
the theory of prolate spheroidal wave functions (also called Slepian functions) and it analyzes the energy preservation problems. Quaternion linear canonical transforms (QLCTs) \cite{kou2013uncertainty} are a family of integral transforms, which generalized the QFT and quaternion fractional Fourier transform (QFRFT) \cite{guanlei2008fractional, guo2011reduced, chen2015pitt, kou2014asymptotic, yang2014uncertainty}.
Some important properties of  QLCTs,  such as convolution and Parseval theorems have been studied in \cite{guanlei2008fractional, guo2011reduced, chen2015pitt, hitzer2007quaternion, kou2013uncertainty, kou2014asymptotic, pei2001efficient, yang2014uncertainty, bahri2013convolution, de2015connecting}. Some studies \cite{hitzer2007quaternion, kou2014asymptotic, pei2001efficient} were briefly introduced the inversion theorem for QLCT and QFT, without a clear proof on his existence. In \cite{guanlei2008fractional}, authors proved  the reversibility of QFRFT, without clearly states the conditions of the existence. Motivating  of the above study, the main purpose of this paper is to solve the following two problems.

\begin{itemize}
\item {\bf Problem A:} If 2D quaternionic-valued integrable function $f$ is a quaternion bounded variation function (please refer to Definition \ref{def312}), then  can $f$ be recovered from their QFTs or/and QLCTs functions?
\item {\bf Problem B:} If 2D quaternionic-valued function $f$ and its QFTs or/and QLCTs  are both integrable, then can $f$ be recovered from their QFTs or/and QLCTs?
\end{itemize}
We notice that the solutions of these two problems have not been carrying out in the literature.
The outline of the paper is as follows. In order to make it self-contained, in Section \ref{sec2}, we collect some basic concepts of quaternionic analysis, QFTs, QLCTs and the 2D real functions of bounded variation to be used throughout the paper. We prove
 the inversion  theorems of  QFTs  and  QLCTs under different conditions for integrable functions in Section \ref{sec3}.
 Some conclusions are drawn, and future works are proposed in section \ref{sec4}.

\section{Preliminary}\label{sec2}

The present section collects some basic facts about quaternions, QFTs, QLCTs, and 2D bounded variation function, which will be needed throughout this paper.

\subsection{The QFTs and QLCTs}
\mbox{}\indent

 Let $\mathbb{H}$ denote the {\it Hamiltonian skew field of quaternions}:
\begin{eqnarray}
\mathbb{H} := \{q=q_0+\i q_1+ \j q_2+\k q_3 \, | \, q_0, q_1, q_2, q_3\in\mathbb{R}\} ,\label{hu1}
\end{eqnarray}
which is an associative non-commutative four-dimensional algebra. The basis elements $\{\i, \j, \k \}$ obey the Hamilton's multiplication rules:
 $\i^2=\j^2=\k^2=-1$, $\i \j =-\j \i =\k$, $\j \k =-\k \j =\i$ and $\k \i=-\i \k =\j$. In this way the quaternionic algebra arises as a natural extension of the complex field $\mathbb{C}$.
In this paper, the complex field $\mathbb{C}$ can be regarded as the 2D plane which is spanned by $ \{ 1, \i\}$.
The {\it quaternion conjugate} of a quaternion $q$ is defined by
\begin{eqnarray*}
\overline{q}=q_0- \i q_1- \j q_2- \k q_3,\quad q_0, q_1, q_2, q_3\in\mathbb{R}.\label{hu2}
\end{eqnarray*}
The modulus of $q\in\mathbb{H}$  is defined as
\begin{eqnarray*}
|q| = \sqrt{q\overline{q}} = \sqrt{\overline{q}q} = \sqrt{q_0^2+q_1^2+q_2^2+q_3^2}.
\end{eqnarray*}
It is not difficult to see that
\begin{eqnarray*}
\overline{qp}=\overline{p}\, \overline{q},\quad |q|=|\overline{q}|,\quad |qp|=|q||p|,\quad \forall \, q,p \in\mathbb{H}.
\end{eqnarray*}

By the Equation (\ref{hu1}), a quaternionic-valued function  $f:\mathbb{R}^2\to\mathbb{H}$ can be expressed in the following form:
\begin{eqnarray*}
f(s,t)=f_0(s,t)+\i f_1(s,t)+\j f_2(s,t)+ \k f_3(s,t),\label{liu209}
\end{eqnarray*}
where $f_{n}\in\mathbb{R}, n=0,1,2,3$.

Let $L^{p}(\mathbb{R}^2 , \mathbb{H}),$ ( integers $ p \geq$ 1) be  the right-linear quaternionic-valued Banach space in $\mathbb{R}^2,$ whose quaternion modules are
defined as follows:
\begin{eqnarray*}
L^{p}(\mathbb{R}^2, \mathbb{H}):=\left\{ f \,| \, f:\mathbb{R}^2 \to \mathbb{H}, \Vert f \Vert_{p}:=\left (\int_{\mathbb{R}^2} |f(s,t)|^p dsdt \right )^{\frac{1}{p}} <\infty \right\}.
\end{eqnarray*}

Due to the non-commutative property of multiplication of quaternions, there are different types of QFTs \cite{hitzer2007quaternion} and  QLCTs \cite{kou2013uncertainty}, respectively. The $\textbf{two-sided}$  QFT:
\begin{eqnarray}\label{twosidedqft}
\mathcal{F}_{T}(u, v):=\int_{\mathbb{R}^2} e^{-{ \i}us}f(s,t)e^{{- \j}vt}dsdt.\label{hu24}
\end{eqnarray}
The $\textbf{right-sided}$   QFT:
\begin{eqnarray*}
\mathcal{F}_{R}(u, v):=\int_{\mathbb{R}^2} f(s,t)e^{-{ \i}us}e^{-{ \j}vt}dsdt.\label{hu25}
\end{eqnarray*}
The $\textbf{left-sided}$   QFT:
\begin{eqnarray*}
\mathcal{F}_{L}(u, v):=\int_{\mathbb{R}^2} e^{-{ \i}us}e^{-{ \j}vt}f(s,t)dsdt.\label{h25}
\end{eqnarray*}
\par
 The QLCTs are the generalization of QFTs,
let $A_{i}=\left(
             \begin{array}{cc}
               a_{i} &b_{i} \\
               c_{i}  &d_{i}  \\
             \end{array}
           \right)
 \in \mathbb{R}^{2\times 2} $
 be  real matrixes  parameter with unit determinant, i.e. $det(A_{i})$=$ a_{i}d_{i}-c_{i}b_{i}=1, $ for $i=1,2$.

\begin{eqnarray}\label{hu26}
 \mathcal{L}_{T}^{ \i,\j}(f)(u,v):=  \left\{
\begin{array}{llll}
\int_{\mathbb{R}^2} K_{A_{1}}^{ \i}(s,u)f(s,t)K_{A_{2}}^{ \j}(t,v)dsdt
& b_{1}, b_{2}\neq 0,\\[1.5ex]
\int_{\mathbb{R}^2} \sqrt{d_{1}} e^{\i\frac{c_{1}d_{1}u^{2}}{2}}f(du,t)K_{A_{2}}^{\j}(t,v)dt
 & b_{1}=0, b_{2}\neq 0, \\[1.5ex]
\int_{\mathbb{R}^2} K_{A_{1}}^{\i}(s,u)f(s,dv)\sqrt{d_{2}} e^{\j\frac{c_{2}d_{2}v^{2}}{2}}ds
& b_{1}\neq 0, b_{2} =0, \\[1.5ex]
\sqrt{d_{1}} e^{\i\frac{c_{1}d_{1}u^{2}}{2}}f(du,dv)\sqrt{d_{2}} e^{\j\frac{c_{2}d_{2}v^{2}}{2}}
& b_{1}= 0, b_{2} =0.,
\end{array}\right.
\end{eqnarray}
 where  the kernels $K_{A_{1}}^{ \i}$ and $K_{A_{2}}^{ \j}$ of  the QLCT are given by
\begin{eqnarray*}
K_{A_{1}}^{\i}(s,u):=\frac{1}{ \sqrt{\i2\pi b_{1}}}e^{\i(\frac{a_{1}}{2b_{1}}s^{2}-\frac{1}{b_{1}}us+\frac{d_{1}}{2b_{1}}u^{2} )}
 \quad and \quad
K_{A_{2}}^{\j}(t,v):=\frac{1}{\sqrt{\j2\pi b_{2}}}e^{\j(\frac{a_{2}}{2b_{2}}t^{2}-\frac{1}{b_{2}}tv+\frac{d_{2}}{2b_{2}}v^{2} )},
\end{eqnarray*} respectively.
Note that when $b_i=0$ $(i=1,2)$, the QLCT of a function is essentially
a chirp multiplication and is of no particular interest for our objective interests.
Hence, without loss of generality, we set
$b_i >0$ $(i=1,2)$ throughout the paper.


Let $\mathcal{L}_{R}^{ \i,\j}(f)(u,v)$ and $\mathcal{L}_{L}^{ \i,\j}(f)(u,v)$ be the $ \textbf{right-sided}$ and $\textbf{left-sided}$  QLCTs, respectively. They are defined by
\begin{eqnarray*}
 \mathcal{L}_{R}^{ \i,\j}(f)(u,v):=
\int_{\mathbb{R}^2}f(s,t)K_{A_{1}}^{ \i}(s,u)K_{A_{2}}^{ \j}(t,v)dsdt
\end{eqnarray*}
and
\begin{eqnarray*}
 \mathcal{L}_{L}^{ \i,\j}(f)(u,v):=
\int_{\mathbb{R}^2}K_{A_{1}}^{ \i}(s,u)K_{A_{2}}^{ \j}(t,v)f(s,t)dsdt,
\end{eqnarray*} respectively.

It is significant to note that the QLCT converts to its special cases while we take different matrices $A_{n}, n=1,2 $  \cite{guanlei2008fractional,pei2001efficient}. For example,
when  $A_{1}=A_{2}=\left(
             \begin{array}{cc}
               0 &1 \\
               -1  &0  \\
             \end{array}
           \right)$,
the QLCT   reduces to the QFT times $ \sqrt{\frac{-\i}{2\pi}}$ and  $\sqrt{\frac{-\j}{2\pi}}$,
where $ \sqrt{-\i}=e^{-\i\pi/4}$ and  $\sqrt{-\j}=e^{-\j\pi/4}$.
 If $ A_{1}=\left(
             \begin{array}{cc}
               \cos \alpha & \sin \alpha \\
               -\sin \alpha  & \cos\alpha  \\
             \end{array}
           \right)
, A_{2}=\left(
             \begin{array}{cc}
               \cos \beta & \sin \beta \\
               -\sin \beta & \cos\beta  \\
             \end{array}
           \right)$,
 the QLCT  becomes the QFRFT  multiplied with the fixed phase factors $ e^{-\i\alpha/2}, e^{-\j\beta/2} $.
\par

\subsection{2D Real Bounded Variation Functions Revisited}

In 1881, Jordan \cite{jordan1881serie} introduced the 1D bounded variation  functions and applied them to the Fourier theory. Hereinafter 1D bounded variation  functions was generalized  to 2D bounded variation  functions  by many authors, for instance \cite{clarkson1933definitions, adams1934properties}.  In the following, we apply definition by Hardy \cite{hardy1906double}, it is a natural generalization of  1D bounded variation functions.  Many important properties are analogous to the 1D case, such as the well-known Dirichlet-Jordan theorem \cite{zygmund2002trigonometric}. It sates that the Fourier series of a bounded variation function $f$ converges at almost every point $x \in [0, 2\pi]$ to the value $ \frac{f(x+0)+f(x-0)}{2}$. Hardy \cite{hardy1906double} generalized  the theorem to the double Fourier series case. In this subsection, we reveiw some properties of 2D bounded variation functions $(BVFs)$, which are required for the subsequent derivations. For a more detailed presentation, please refer to  \cite{adams1934properties,clarkson1933definitions}.

The 2D real  function $f$ is assumed to be defined in a rectangle $\mathbb{E}$ $(a_{1} \leq  s \leq b_{1}, a_{2}\leq t \leq b_{2})$. By the term net we shall, unless otherwise specified, mean a set of parallels to the axis:
$ a_{1}=s_{0}<s_{1}< \cdots <s_{m}=b_{1}, $
and
$  a_{2}=t_{0}<t_{1}< \cdots <t_{n}=b_{2}. $

Each of the smaller rectangles into which $\mathbb{E}$ is divided by a net will be called a cell, we employ the notations
 $$ \triangle _{10}f(s_{i},t_{j}):=f(s_{i+1},t_{j})-f(s_{i},t_{j}),$$
  $$ \triangle _{01}f(s_{i},t_{j}):=f(s_{i},t_{j+1})-f(s_{i},t_{j}),$$
    $$ \triangle _{11}f(s_{i},t_{j})=\triangle _{10} ( \triangle _{01} f)(s_{i},t_{j}):
    =f(s_{i+1},t_{j+1})-f(s_{i+1},t_{j})-f(s_{i},t_{j+1})+f(s_{i},t_{j}).$$

 \begin{defn}\label{De31}
 \cite{clarkson1933definitions}
Let $f: [a_{1},  b_{1}]\times [a_{2}, b_{2}] \rightarrow \mathbb{R} $ be said to be a bounded variation function $(BVF)$, if it satisfies the following conditions:
 \begin{itemize}
   \item   the sum $\sum_{i=0}^{m-1}\sum_{j=0}^{n-1}\big |\triangle _{11}f(s_{i},t_{j})\big |$ is bounded for all nets,
   \item   $f(\tilde{s},t)$ considered as a function of $t$ alone in the interval  $[a_{2}, b_{2}]$ is of 1D bounded variation
   for at least one $ \tilde{s} $ in $[a_{1},  b_{1}]$,
   \item  $f(s, \tilde{t})$ considered as a function of $s$ alone in the interval $[a_{1},  b_{1}]$ is of 1D bounded variation
   for at least one $\tilde{t} $ in $[a_{2}, b_{2}]$.
 \end{itemize}
 \end{defn}

 \begin{remark}\label{re32}
From Definition $\ref{De31}$, it follows that if $f$ and $g$ are both $BVFs$,  then  $f \pm g$ and $f g$ are also BVFs.
\end{remark}

The well-known  Jordan decomposition Theorem states that a 1D real function $f$ is bounded variation if it can be written as a difference of two monotone increasing functions. The 2D $BVF$ in Definition $\ref{De31}$ has the analogue result.  Before to proceed, we need the following definition of 2D monotone increasing function.

  \begin{defn}\label{de33}\cite{clarkson1933definitions}
 The 2D real function $f$ which satisfies the following conditions everywhere in its domain is called a quasi-monotone function:
 \begin{itemize}
   \item  $ \triangle _{11}{_{(s_{0},t_{0})}^{(s_{1},t_{1})}}f(s,t):=f(s_{1},t_{1})-f(s_{1},t_{0})-f(s_{0},t_{1})+f(s_{0},t_{0})\geq 0, \mbox{for } s_{1}
   \geq  s_{0} , t_{1} \geq t_{0} $,
   \item  $f(s,t)$ is monotone and non-diminishing with respect for $s$, for every constant value of $t$,
   \item   $f(s,t)$ is monotone and non-diminishing with respect for $t$, for every constant value of $s$.
 \end{itemize}
 \end{defn}

  \begin{lemma}\cite{hobson1907theory}\label{le38}
If $f(s,t)$ is a quasi-monotone function,
then four double limits $$ f(s- 0,t-0 )=\lim_{\begin{subarray}
hh \to 0^{+}\\k \to 0^{+}
\end{subarray}}f(s- h,t-k ),  f(s+ 0,t-0 )=\lim_{\begin{subarray}
hh \to 0^{+}\\k \to 0^{+}
\end{subarray}}
 f(s+ h,t-k ),$$  $$f(s- 0,t+0 )=\lim_{\begin{subarray}
hh \to 0^{+}\\k \to 0^{+}
\end{subarray}}f(s- h,t+k ),  f(s+0,t+0 )=\lim_{\begin{subarray}
hh \to 0^{+}\\k \to 0^{+}
\end{subarray}}f(s+ h,t+k ), $$
 all exist and have definite numbers.
\end{lemma}
 \begin{example}\label{ex34}
 $f(s,t)=e^{s}e^{t}$ is a quasi-monotone function.
 \end{example}
 \par
The quasi-monotone function has lots of good properties, please refer to \cite{hobson1907theory} for more detail.
From the  Definition $\ref{de33}$, it leads to the fact that if $f(s,t)$ is  bounded quasi-monotone function then  $f(s,t)$ is  the $BVF.$

\begin{lemma} \label{le36}\cite{adams1934properties}
 A necessary and sufficient condition that $f(s,t)$ is  the $BVF$ is that it can be represented as the difference between two
 bounded functions, $f_{1}(s,t)$ and $f_{2}(s,t)$, satisfying the inequalities
 $$\triangle _{11}f_{i}(s,t)\geq 0,\triangle _{01}f_{i}(s,t)\geq 0,\triangle _{10}f_{i}(s,t)\geq 0, i=1,2.$$
\end{lemma}

\begin{remark}\label{re37}
From  Definition $\ref{de33}$ and Lemma $\ref{le36}$, we can generalize the Jordan decomposition Theorem for BVF from one dimension to two dimension. If $f(s,t)$ is the $BVF$ if and only if it can be represented  as the
difference between two bounded quasi-monotone functions $f_{1}$ and $f_{2}$, i.e  $f= f_{1}-f_{2}$ .

\end{remark}

The mean value theorem pay a important role in 1D Fourier theory.
For the higher dimensional cases, we present 2D extension  as follows.

\begin{lemma}\cite{hobson1907theory}\label{le39}
 Let $\mathbb{E}$ be the plane rectangle $ [a_{1} ,a_{2}]\times [b_{1},  b_{2}]$,
 the non-negative function $f(s,t)$ is of  quasi-monotone in $\mathbb{E}$,  and $g(s,t)$ is summable in $\mathbb{ E}$.
Then
\begin{eqnarray*}
 \int_{a_{1}}^{b_{1}}\int_{ a_{2}}^{ b_{2}}f(s,t)g(s,t)dsdt=f(b_{1}-0 , b_{2}-0)\int_{\xi_{1} }^{b_{1}}\int_{ \xi_{2}}^{ b_{2}}g(s,t)dsdt,
 \end{eqnarray*}
for some point $(\xi_{1} ,\xi_{2}) $ in $ \mathbb{E} $.
\end{lemma}
\setcounter{equation}{0}

\setcounter{equation}{0}
\section{Main results}\label{sec3}
In this  section, we firstly focus on  the inversion  theorem of 2D Fourier transform $(FT)$ by using  the  properties  of the BVFs,   then  we drive the inversion 2D QFTs  and QLCTs theorems.
In the following, the cross-neighborhood of $ (x_{0},y_{0})$ \cite{hobson1907theory} defined by the set of pinits
\begin{eqnarray*}
|s-x_{0}|\leq \varepsilon_{1},\quad or \setminus and  \quad  |t-y_{0}|\leq \varepsilon_{2}, \quad (\varepsilon_{i}>0,i=1,2 )
\end{eqnarray*}

\begin{defn}
$f(s,t) $ is called to belong to $ \bf{L}$ class $( \bf{LC})$ in the cross- neighborhood of $(x_{0},y_{0})$, if f satisfies:
$$
\int_{\varepsilon_{2}}^{\infty}\int_{0}^{\varepsilon_{1}} \left | \frac{ \tilde{f}(s,t)-\tilde{f}(a,t)}{s} \right |dsdt < \infty
$$
and
$$
\int_{\varepsilon_{1}}^{\infty}\int_{0}^{\varepsilon_{2}} \left | \frac{ \tilde{f}(s,t)-\tilde{f}(s,b)}{t} \right |dtds < \infty,
$$
where
$\tilde{f}(s,t)=f(x_{0}-s, y_{0}-t)+ f(x_{0}+s, y_{0}+t)+f(x_{0}-s, y_{0}+t)+f(x_{0}+s, y_{0}-t),  \varepsilon_{n}>0,$ n=1,2, and $a\in\mathbb{ A}:=\{ s \in \mathbb{R}| \int_{\mathbb{R}}| \tilde{f}(s,t)|dt < \infty \}, b\in \mathbb{B}:=\{ t \in \mathbb{R}| \int_{\mathbb{R}}| \tilde{f}(s,t)|ds < \infty \}.$
\end{defn}
 \begin{remark}
 If $\int_{\mathbb{R}^{2}} |\tilde{f}(s,t)|dsdt < \infty$, the Fubini theorem implies that $ \int_{\mathbb{R}}| \tilde{f}(s,t)|dt < \infty $  holds for $s$ almost everywhere, and  $ \int_{\mathbb{R}}| \tilde{f}(s,t)|ds < \infty $  holds for $t$ almost everywhere, then $\mathbb{A}$ and $\mathbb{B}$ are the real line except a measurable zero set.

 \begin{example}
 If $f(s,t)=f_{1}(s)f_{2}(t)\in L^{1}(\mathbb{R}^{2}, \mathbb{R} ), \frac{\partial f(s,t)}{\partial s}|_{s=x_{0}} $ and $\frac{\partial f(s,t)}{\partial s}|_{t=y_{0}} $ exist,
 then $f(s,t) \in \bf{LC} $.
 \end{example}
 \end{remark}

\begin{remark}
In this paper, we use the swash capital  and capital to denote the QFT and 2D FT  respectively.
\end{remark}
\subsection{Problem  A for QFTs}
 In this subsection, first, by using the good properties of the BVF, we  drive  inversion theorem for 2D FT, the situation is by the similar argument as the proof of the convergence of two-dimensional Fourier series  \cite{hardy1906double},
Second,  the  bounded  variation function is defined  in quaternion fields, and  the inversion theorem of two-sided QFT is proved. Finally  we generalize that our idea  to the all types of QFTs.

\begin{lemma}\cite{PWJ2000}\label{le310}
 For any given  real number $ a, b,$ we have
 $$ \left |\int_{a}^{b}\frac {\sin t}{ t}dt \right | \leq 6.$$
\end{lemma}

We first proceed the 2D Fourier inversion theorem of real function.
\begin{theorem}($\textbf{ 2D Fourier  Inversion Theorem}$)\label{th311}
Suppose that $f \in L^{1}(\mathbb{R}^2, \mathbb{R})$,   $f$ is the  BVF and belongs to $\bf{LC}$ in a   cross-neighborhood of $(x_{0},y_{0})$,
then
\begin{eqnarray}\label{h825}
 \eta(x_{0},y_{0})= \lim_{\begin{subarray}
  NN \to \infty\\ M \to \infty
\end{subarray}}\frac{1}{4\pi^2}\int_{-N}^{N}\int_{-M}^{M}F(u,v)e^{\i ux_{0}}e^{\i vy_{0}}dudv,
\end{eqnarray}
where
$$\eta(x_{0},y_{0}):=\frac {f(x_{0}+0,y_{0}+0)+f(x_{0}+0,y_{0}-0)+f(x_{0}-0,y_{0}+0)+f(x_{0}-0,y_{0}-0)}{ 4},$$
and the 2D FT is defined by
\begin{eqnarray}\label{e9}
F(u,v):=\int_{\mathbb{R}^{2}}f(s,t)e^{-\i us}e^{-\i vt}dsdt.
\end{eqnarray}
If $f$ is continuous  at  $(x_{0},y_{0})$,
then $\eta(x_{0},y_{0})=f(x_{0},y_{0})$.

\end{theorem}

\begin{proof}
Since $f \in L^{1}(\mathbb{R}^2, \mathbb{ R})$, then $F(u,v)=\int_{\mathbb{R}^{2}}f(s,t)e^{-\i us}e^{-\i vt}dsdt $ is well defined.
\\
Set
$$I(x_{0},y_{0},N,M):=\frac{1}{4\pi^2}\int_{-N}^{N}\int_{-M}^{M}F(u,v)e^{\i ux_{0}}e^{\i vy_{0}}dudv,$$

inserting the definition of $F(u,v)$, we have
\begin{eqnarray*}
I(x_{0},y_{0},N,M)=&&\frac{1}{4\pi^2}\int_{-N}^{N}\int_{-M}^{M}\int_{-\infty}^{\infty}\int_{-\infty}^{\infty}f(s,t)e^{-\i us}e^{-\i vt}e^{\i ux_{0}}e^{\i vy_{0}}dsdtdudv\\
=&&\frac{1}{4\pi^2}\int_{-\infty}^{\infty}\int_{-\infty}^{\infty}f(s,t)\int_{-N}^{N}\int_{-M}^{M}e^{\i u(x_{0}-s)}e^{\i v(y_{0}-t)}dudvdsdt.\\
\end{eqnarray*}
Switching the order of integration is permitted by the Fubin theorem, because of $ f \in L^{1}(\mathbb{R}^2, \mathbb{R})$.\\
\begin{eqnarray*}
I(x_{0},y_{0},N,M)=&&\int_{-\infty}^{\infty}\int_{-\infty}^{\infty}f(s,t)\frac{\sin [M(x_{0}-s)]}{\pi(x_{0}-s)}\frac{\sin [N(y_{0}-t)]}{\pi(y_{0}-t)}dsdt\\
=&&\int_{-\infty}^{\infty}\int_{-\infty}^{\infty}f(x_{0}-s,y_{0}-t)\frac{\sin (Ms)}{\pi s}\frac{\sin (Nt)}{\pi t}dsdt\\
=&&\int_{0}^{\infty}\int_{0}^{\infty}\big(f(x_{0}-s,y_{0}-t)+f(x_{0}+s,y_{0}-t)\\
&&+f(x_{0}-s,y_{0}+t)+f(x_{0}+s,y_{0}+t)\big)\frac{\sin Ms}{\pi s}\frac{\sin Nt}{\pi t}dsdt.
\end{eqnarray*}

Since
\begin{eqnarray*}
\int_{0}^{\infty}\int_{0}^{\infty}\frac{\sin (Ms)}{\pi s}\frac{\sin (Nt)}{\pi t}dsdt=\frac{1}{4},
\end{eqnarray*}
then
\begin{eqnarray*}
I(x_{0},y_{0},N,M) - \eta(x_{0},y_{0}) =&&
\int_{0}^{\infty}\int_{0}^{\infty}(f(x_{0}-s,y_{0}-t)+f(x_{0}+s,y_{0}-t)+f(x_{0}-s,y_{0}+t)\\
&&+f(x_{0}+s,y_{0}+t)-4\eta(x_{0},y_{0}))\frac{\sin (Ms)}{\pi s}\frac{\sin (Nt)}{\pi t}dsdt.
\end{eqnarray*}
Let

$$\phi_{(x_{0},y_{0})}(s,t):=\tilde{f}(s,t)-4\eta(x_{0},y_{0}),$$
where $\tilde{f}(s,t)=f(x_{0}-s,y_{0}-t)+f(x_{0}+s,y_{0}-t)+f(x_{0}-s,y_{0}+t)+f(x_{0}+s,y_{0}+t). $\\
Then $$ \lim_{\begin{subarray}
tt \to 0^{+}\\s \to 0^{+}
\end{subarray}}\phi_{(x_{0},y_{0})}(s,t)=\tilde{f}(s,t)-4\eta(x_{0},y_{0}) =0.$$

Applying  Remarks \ref{re32} and \ref{re37}, $\phi_{(x_{0},y_{0})}$is also a  BVF and can be expressed as
the difference between two bounded quasi-monotone functions $ h_{1},$ and $ h_{2}.$
i.e., $$\phi_{(x_{0},y_{0})}=h_{1}-h_{2}.$$
Furthermore, the quasi-monotone functions $ h_{1}$ and $h_{2}$  satisfy
$$ \lim_{\begin{subarray}
tt \to 0^{+}\\s \to 0^{+}
\end{subarray}} h_{i}(s,t)=0, \quad i=1,2, \quad respectively.$$
Therefore $ h_{1}(s,t) $ and $h_{2}(s,t)$ are non-negative in the rectangle $ [0, \delta_{1}] \times [0, \delta_{2}]. $
For any $ \varepsilon >0$, there exists $ \epsilon_{1} >0,\epsilon_{2} >0$, such that $ 0 <\epsilon_{1}<\delta_{1}, 0<\epsilon_{2}<\delta_{2},$
$$ 0\leq h_{i}(s,t)<\varepsilon, \quad for  \quad all \quad 0 <s\leq \epsilon_{1},0 <t\leq\epsilon_{2}. $$
Now we divide the  $ [0, \infty ) \times [0, \infty ) $   into four parts  with the $\epsilon_{1},\epsilon_{2}. $
\begin{eqnarray}
I(x_{0},y_{0},N,M) - \eta(x_{0},y_{0}) =&&\int_{0}^{\epsilon_{2}}\int_{0}^{\epsilon_{1}}\left(h_{1}(s,t)-h_{2}(s,t)\right)\frac{\sin (Ms)}{\pi s}\frac{\sin (Nt)}{\pi t}dsdt
\nonumber \\
&&+\int_{\epsilon_{2}}^{\infty}\int_{0}^{\epsilon_{1}}\left(\tilde{f}(s,t)-4\eta(x_{0},y_{0})\right)\frac{\sin (Ms)}{\pi s}\frac{\sin (Nt)}{\pi t}dsdt
\nonumber \\
&&+\int_{0}^{\epsilon_{2}}\int_{\epsilon_{1}}^{\infty}\left(\tilde{f}(s,t)-4\eta(x_{0},y_{0})\right)\frac{\sin (Ms)}{\pi s}\frac{\sin (Nt)}{\pi t}dsdt
\nonumber \\
&&+\int_{\epsilon_{2}}^{\infty}\int_{\epsilon_{1}}^{\infty}\left(\tilde{f}(s,t)-4\eta(x_{0},y_{0})\right)\frac{\sin (Ms)}{\pi s}\frac{\sin (Nt)}{\pi t}dsdt. \label{hu32}
\end{eqnarray}

Let $I_{1}$ denote the first term of the Equation (\ref{hu32}), we can obtain that

\begin{eqnarray*}
|I_{1}| \leq \sum_{i=1}^{2} \frac{2}{\pi^2}\left | \int_{0 }^{\epsilon_{1}}\int_{0}^{ \epsilon_{2}}h_{i}(s,t)\frac{\sin (Ms)}{ s}\frac{\sin (Nt)}{ t}dsdt\right |.
\end{eqnarray*}
By the Lemma \ref{le39}, there exist two points $\big\{(\xi^{(1)}_{i} , \xi^{(2)}_{i}), i=1,2\big \}$ such that

\begin{eqnarray*}
\sum_{i=1}^{2} \frac{2}{\pi^2}\left | \int_{0 }^{\epsilon_{1}}\int_{0}^{\epsilon_{2}}h_{i}(s,t)\frac{\sin (Ms)}{ s}\frac{\sin (Nt)}{ t}dsdt\right|
= \sum_{i=1}^{2}h_{i}(\epsilon_{1},\epsilon_{2}) \frac{2}{\pi^2}\left | \int_{\xi^{(1)}_{i} }^{\epsilon_{1}}\int_{\xi^{(2)}_{i}}^{ \epsilon_{2}}\frac{\sin (Ms)}{ s}\frac{\sin (Nt)}{t}dsdt\right |\\
=\sum_{i=1}^{2}h_{i}(\epsilon_{1},\epsilon_{2}) \frac{2}{\pi^2}\left | \int_{\xi^{(1)}_{i} }^{\epsilon_{1}}\int_{\xi^{(2)}_{i}}^{ \epsilon_{2}}\frac{\sin (Ms)}{ s}\frac{\sin (Nt)}{ t}dsdt\right |\\
\leq \sum_{i=1}^{2}h_{i}(\epsilon_{1},\epsilon_{2}) \frac{2}{\pi^2}\left | \int_{M\xi^{(1)}_{i} }^{M\epsilon_{1}}\frac{\sin s}{ s}ds\right|\left|\int_{N\xi^{(2)}_{i}}^{ N\epsilon_{2}}\frac{\sin t}{ t}dsdt\right|.
\end{eqnarray*}
Hence
\begin{eqnarray*}
|I_{1}| \leq (\frac{4}{\pi^2} 6^2)\varepsilon,
\end{eqnarray*}
the last step applying  Lemma \ref{le310}.
\par
Let $I_{2}$ be the second  term  of the Equation $(\ref{hu32}),$

\begin{eqnarray*}
I_{2}&&:=\int_{\epsilon_{2}}^{\infty}\int_{0}^{\epsilon_{1}}(\tilde{f}(s,t)-4\eta(x_{0},y_{0}))\frac{\sin (Ms)}{\pi s}\frac{\sin (Nt)}{\pi t}dsdt \nonumber \\
&&=\int_{\epsilon_{2}}^{\infty}\int_{0}^{\epsilon_{1}}(\tilde{f}(s,t)-\tilde{f}(a,t))\frac{\sin (Ms)}{\pi s}\frac{\sin (Nt)}{\pi t}dsdt
+\int_{\epsilon_{2}}^{\infty}\int_{0}^{\epsilon_{1}} \tilde{f}(a,t) \frac{\sin (Ms)}{\pi s}\frac{\sin (Nt)}{\pi t}dsdt \nonumber \\
&&-\int_{\epsilon_{2}}^{\infty}\int_{0}^{\epsilon_{1}}4\eta(x_{0},y_{0}))\frac{\sin (Ms)}{\pi s}\frac{\sin (Nt)}{\pi t}dsdt. \label{hu34}
\end{eqnarray*}
Since $f(s,t) \in \bf{LC}$, then due to  the Riemanm-Lebesque lemma \cite{stein1971introduction}, the first term of $ I_{2} \rightarrow 0,$ as  N $\rightarrow \infty $.\\
For the  second term of $ I_{2}$,
$$ \left |\int_{\epsilon_{2}}^{\infty}\int_{0}^{\epsilon_{1}} \tilde{f}(a,t) \frac{\sin (Ms)}{\pi s}\frac{\sin (Nt)}{\pi t}dsdt \right |\leqslant
\left | \int_{\epsilon_{2}}^{\infty}  \tilde{f}(a,t)\frac{\sin (Nt)}{\pi t}dt \right| \left | \int_{0}^{\epsilon_{1}}  \frac{\sin (Ms)}{\pi s}ds   \right|, $$
where $\tilde{f}(a,\cdot) \in L^{1}(\mathbb{R}, \mathbb{R})$ because of $  f(s,t) \in \bf{LC}$, then by the Riemanm-Lebesque lemma \cite{stein1971introduction}, the second term of $ I_{2} \rightarrow 0,$ as N $\rightarrow \infty $.

$$ \left |\int_{\epsilon_{2}}^{\infty}\int_{0}^{\epsilon_{1}}4\eta(x_{0},y_{0})\frac{\sin (Ms)}{\pi s}\frac{\sin (Nt)}{\pi t}dsdt \right |
\leqslant \left | 4\eta(x_{0},y_{0}) \int_{0}^{\epsilon_{1}}  \frac{\sin (Ms)}{\pi s} ds  \right| \left |\int_{\epsilon_{2}}^{\infty} \frac{\sin(Nt)}{\pi t}dt \right |,$$
 since $ \int_{\epsilon_{2}}^{\infty} \frac{\sin (Nt)}{\pi t}dt \rightarrow 0,$ as N $\rightarrow 0$, we have the third  term of $ I_{2} \rightarrow 0,$ as N
 $\rightarrow \infty .$
 \par
Let $I_{3}$ denote the third  term  of the Equation (\ref{hu32}).
$$I_{3}:=\int_{\epsilon_{2}}^{\infty}\int_{0}^{\epsilon_{1}}(\tilde{f}(s,t)-4\eta(x_{0},y_{0}))\frac{\sin (Ms)}{\pi s}\frac{\sin (Nt)}{\pi t}dsdt,$$
taking  similarly argument as in  $I_{2}$,  imply that $I_{3} \to 0 $, as M $\rightarrow \infty $.
\par
Let $I_{4}$ denote the fourth  term  of the Equation (\ref{hu32}).
\begin{eqnarray*}
I_{4}:=&&\int_{\epsilon_{2}}^{\infty}\int_{\epsilon_{1}}^{\infty}(\tilde{f}(s,t)-4\eta(x_{0},y_{0}))\frac{\sin (Ms)}{\pi s}\frac{\sin (Nt)}{\pi t}dsdt \nonumber \\
=&&\int_{\epsilon_{2}}^{\infty}\int_{\epsilon_{1}}^{\infty}\tilde{f}(s,t)\frac{\sin (Ms)}{\pi s}\frac{\sin (Nt)}{\pi t}dsdt-\int_{\epsilon_{2}}^{\infty}\int_{\epsilon_{1}}^{\infty}4\eta(x_{0},y_{0})\frac{\sin (Ms)}{\pi s}\frac{\sin (Nt)}{\pi t}dsdt, \label{hu36}
\end{eqnarray*}
since $\tilde{f}\in L^{1}(\mathbb{R}^2, \mathbb{R})$  and $\frac{1}{\pi t} \leq \frac{1}{\pi \epsilon_{2}},\frac{1}{\pi s} \leq \frac{1}{\pi \epsilon_{1}}$, as  $ s \in (\epsilon_{1 } , \infty),t \in (\epsilon_{2 }, \infty),$
by the Riemanm-Lebesque Lemma \cite{stein1971introduction}, we can conclude that $I_{4} \to 0 $ ,  as M $ \rightarrow \infty,$ N $\rightarrow \infty .$
We conclude the proof,
i.e..
$$I(x_{0},y_{0},N,M) - \eta(x_{0},y_{0})  \to  0, as  \quad M \to \quad \infty,\quad N \to \quad \infty. $$

\end{proof}
\begin{remark}
For separable function $f(s,t)=f_{1}(s)f_{2}(t)$ is a BVF in a rectangle-neighborhood of $(x_{0},y_{0})$ and belongs to $L^{1}(\mathbb{R}^{2}, \mathbb{R} )$, then the Equation $(\ref{h825})$ is also valid, that is to say, the  $ \mathbf{LC}$ condition is not required in same cases.
\end{remark}

\begin{defn}\label{def312}
 The function $f(s,t)=f_{0}(s,t)+\i f_{1}(s,t)+\j f_{2}(s,t)+\k f_{3}(s,t)$ is said to be \textbf{ quaternion bounded variation function $(QBVF)$ } if and only if it's components $f_{n}(s,t),n=0,1,2,3$ are all  BVFs.
\end{defn}

\begin{lemma}\label{le313}
If $f \in L^{1}(\mathbb{R}^2,\mathbb{H})$, if and only if  its components $f_{n} \in L^{1}(\mathbb{R}^2, \mathbb{R}), n=0,1,2,3.$
\end{lemma}
\begin{proof}

The module $|f(s,t)|$ of a quaternionic-valued function  $f(s,t)$ is given by
\begin{eqnarray*}
|f(s,t)|  = \sqrt{|f_0(s,t)|^2+|f_1(s,t)|^2+|f_2(s,t)|^2+|f_3(s,t)|^2}.
\end{eqnarray*}
Therefore, if $ f \in L^{1}(\mathbb{R}^2,\mathbb{H}),$
then
\begin{eqnarray*}
\int_{\mathbb{R}^2} \left|f_{n}(s,t)\right|dsdt \leq \int_{\mathbb{R}^2} |f(s,t)|dsdt < \infty,
\end{eqnarray*}
$n=0,1,2,3.$
\par
On the other hand, since
\begin{eqnarray*}
|f(s,t)|  \leq  |f_0(s,t)|+|f_1(s,t)|+|f_2(s,t)|+|f_3(s,t)|,
\end{eqnarray*}
then
\begin{eqnarray*}
\int_{\mathbb{R}^2} |f(s,t)|dsdt \leq  \sum_{n=0}^{3}\int_{\mathbb{R}^2} |f_{n}(s,t)|dsdt < \infty.
\end{eqnarray*}
\end{proof}
\begin{theorem}($\textbf{ Inversion Theorem for two-sided QFT}$)\label{the314}

Suppose in the cross-neighborhood  of $ (x_{0},y_{0})$, $f(s,t)$ is  the  QBVF  and belongs to $\bf{LC},$  and $f \in L^{1}(\mathbb{R}^2,\mathbb{H}) $,
then
\begin{eqnarray}\label{h81}
\eta(x_{0},y_{0})= \lim_{\begin{subarray}

N \to \infty \\M \to \infty
\end{subarray}
 }\frac{1}{4\pi^2}\int_{-N}^{N}\int_{-M}^{M}e^{\i ux_{0}}\mathcal{F}_{T}(u,v)e^{\j vy_{0}}dudv,
 \end{eqnarray}
where
\begin{eqnarray*}
\eta(x_{0},y_{0})&&=\frac {f(x_{0}+0,y_{0}+0)+f(x_{0}+0,y_{0}-0)+f(x_{0}-0,y_{0}+0)+f(x_{0}-0,y_{0}-0)}{ 4}\\
&&=\eta_{0}+\i\eta_{1}+\j\eta_{2}+\k\eta_{3},\\
\eta_{n}(x_{0},y_{0})&&=\frac {f_{n}(x_{0}+0,y_{0}+0)+f_{n}(x_{0}+0,y_{0}-0)+f_{n}(x_{0}-0,y_{0}+0)+f_{n}(x_{0}-0,y_{0}-0)}{ 4},\quad n=0,1,2,3,\\
\mathcal{F}_{T}(u,v)&&=\int_{\mathbb{R}^{2}}e^{-\i us}f(s,t)e^{-\j vt}dsdt.
\end{eqnarray*}
If $f(s,t)$ is continuous  at  $(x_{0},y_{0})$,
then $\eta(x_{0},y_{0})=f(x_{0},y_{0})$.

\end{theorem}
\begin{proof}

Set
$$ \nu(x_{0},y_{0},N,M):=\frac{1}{4\pi^2}\int_{-N}^{N}\int_{-M}^{M}e^{\i ux_{0}}\mathcal{F}_{T}(u,v)e^{\j vy_{0}}dudv,$$

and rewrite this expression by inserting the definition of $\mathcal{F}_{T}(u,v)$:
\begin{eqnarray}\label{sin1}
&&\nu(x_{0},y_{0},N,M)=\frac{1}{4\pi^2}\int_{-N}^{N}\int_{-M}^{M}e^{\i ux_{0}}\int_{\mathbb{R}^{2}}e^{-\i us}f(s,t)e^{-\j vt}dsdte^{\j vy_{0}}dudv \nonumber \\
&=&\int_{\mathbb{R}^{2}}\int_{-N}^{N}\int_{-M}^{M}e^{\i u(x_{0}-s)}f(s,t)e^{\j v(y_{0}-t)}dudvdsdt \nonumber \\
&=&\int_{\mathbb{R}^{2}}\frac{\sin M(x_{0}-s)}{\pi(x_{0}-s)}f(s,t)\frac{\sin N(y_{0}-t)}{\pi(y_{0}-t)}dsdt \nonumber \\
&=&\int_{\mathbb{R}^{2}}f(x_{0}-s,y_{0}-t)\frac{\sin (Ms)}{\pi s}\frac{\sin(Nt)}{\pi t}dsdt \\
&=&\int_{\mathbb{R}^{2}}f(x_{0}-s,y_{0}-t)\frac{\sin (Ms)}{\pi s}\frac{\sin (Nt)}{\pi t}dsdt\nonumber \\
&=&\int_{\mathbb{R}^{2}}f_{0}(x_{0}-s,y_{0}-t)\frac{\sin (Ms)}{\pi s}\frac{\sin (Nt)}{\pi t}dsdt
+\i \int_{\mathbb{R}^{2}}f_{1}(x_{0}-s,y_{0}-t)\frac{\sin (Ms)}{\pi s}\frac{\sin( Nt)}{\pi t}dsdt \nonumber \\
&&+\j\int_{\mathbb{R}^{2}}f_{2}(x_{0}-s,y_{0}-t)\frac{\sin (Ms)}{\pi s}\frac{\sin( Nt)}{\pi t}dsdt
+\k \int_{\mathbb{R}^{2}}f_{3}(x_{0}-s,y_{0}-t)\frac{\sin (Ms)}{\pi s}\frac{\sin (Nt)}{\pi t}dsdt.\nonumber
\end{eqnarray}
Switching the order of integration is permitted by the Fubin Theorem in the first step.\\
Set
\begin{eqnarray*}
\nu_{n}(x_{0},y_{0},N,M):=\int_{\mathbb{R}^2}f_{n}(x_{0}-s,y_{0}-t)\frac{\sin (Ms)}{\pi s}\frac{\sin( Nt)}{\pi t}dsdt, n=0,1,2,3,
\end{eqnarray*}
then
\begin{eqnarray*}
\left|\nu(x_{0},y_{0},N,M)-\eta(x_{0},y_{0}) \right|= \sqrt{\sum_{n}^{4}|\nu_{n}(x_{0},y_{0},N,M)-\eta_{n}(x_{0},y_{0})|^2}.
\end{eqnarray*}
According to   Lemma \ref{le313} and Theorem \ref{th311}, we have
$$|\nu_{n}(x_{0},y_{0},N,M)-\eta_{n}(x_{0},y_{0})| \to 0,\quad  as \quad N \to \infty, M \to \infty, \quad n=0,1,2,3,$$
hence $|\nu(x_{0},y_{0},N,M)-\eta(x_{0},y_{0})|\to 0,$ as  N  $\rightarrow \infty,$  M $\rightarrow \infty$.

\end{proof}

Using  the similar  argument, and the fact that the Sinc function in Equation $(\ref{sin1})$ is real-valued, it commutes with quaternionic-value  function $f,$ then we obtain the inversion formulas for left-sided and right-sided QFTs, respectively.

\begin{theorem}\label{cor315}
 Suppose in the cross-neighborhood    of $(x_{0},y_{0})$, $f(s,t)$ is the QBVF and belongs to $ \bf{LC}$, and $f\in L^{1}(\mathbb{R}^2,\mathbb{H})$,
then

\begin{description}
  \item[(a)] the inversion theorem of right-sided QFT:
   \begin{eqnarray} \label{r1}
\eta(x_{0},y_{0})= \lim_{\begin{subarray}

N \to \infty \\M \to \infty
\end{subarray}
 }\frac{1}{4\pi^2}\int_{-N}^{N}\int_{-M}^{M}\mathcal{F}_{R}(u,v)e^{\j vy_{0}}e^{\i ux_{0}}dudv,
 \end{eqnarray}

  \item[(b)] the inversion theorem of left-sided QFT:
  \begin{eqnarray} \label{l1}
\eta(x_{0},y_{0})= \lim_{\begin{subarray}

N \to \infty \\M \to \infty
\end{subarray}
 }\frac{1}{4\pi^2}\int_{-N}^{N}\int_{-M}^{M}e^{\j vy_{0}}e^{\i ux_{0}}\mathcal{F}_{L}(u,v)dudv,
 \end{eqnarray}

\end{description}

where
\begin{eqnarray*}
\eta(x_{0},y_{0})&&=\frac {f(x_{0}+0,y_{0}+0)+f(x_{0}+0,y_{0}-0)+f(x_{0}-0,y_{0}+0)+f(x_{0}-0,y_{0}-0)}{ 4},\\
&&=\eta_{0}+\i\eta_{1}+\j\eta_{2}+\k\eta_{3},\\
\eta_{n}(x_{0},y_{0})&&=\frac {f_{n}(x_{0}+0,y_{0}+0)+f_{n}(x_{0}+0,y_{0}-0)+f_{n}(x_{0}-0,y_{0}+0)+f_{n}(x_{0}-0,y_{0}-0)}{ 4},\\
&&n=0,1,2,3,\\
\end{eqnarray*}
if $f(s,t)$ is continuous  at  $(x_{0},y_{0})$,
then $f(x_{0},y_{0})=\eta(x_{0},y_{0}) .$

\end{theorem}

\begin{remark}\label{re316}
We shall note that  the two-sided QFT defined above can be generalized as follows:
\begin{eqnarray}\label{genqft}
\mathcal{F}_{T}(u,v):=\int_{\mathbb{R}^2}e^{-\mu_{1}us}f(s,t)e^{-\mu_{2}vt}dsdt,
 \end{eqnarray}
or
\begin{eqnarray*}
\mathcal{F}_{T}(u,v):=\int_{\mathbb{R}^2}e^{-\mu_{1}us}f(s,t)e^{-\mu_{1}vt}dsdt,
 \end{eqnarray*}
where $\mu_{1}:= \mu_{1,1} \i +\mu_{1,2} \j +\mu_{1,3} \k$ and $\mu_2 := \mu_{2,1} \i +\mu_{2,2} \j +\mu_{2,3} \k$ so that

\begin{eqnarray}
&&\mu_{n}=\mu_{n.1}\i+\mu_{n.2}\j+\mu_{n.3}\k;  \label{hu37}\\
&&\mu_{n}^{2}=-\mu_{n.1}^{2}-\mu_{n.2}^{2}-\mu_{n.3}^{2}=-1,n=1,2,\nonumber\\
&&\mu_{1.1}\mu_{2.1}+\mu_{1.2}\mu_{2.2}+\mu_{1.3}\mu_{2.3}=0.\nonumber
\end{eqnarray}

Equation (\ref{twosidedqft}) is the special case of
(\ref{genqft}) in which $\mu_1=\i$ and $\mu_2=\j$.  The right-sided and left-sided QFTs can be also generalized  similarly as above.


Since $\int_{-N}^{N}e^{\mu_{n}x}dx=2\sin(N),N>0,n=1,2$, by the similar argument, we could find that if the  quaternionic-value function $f(s,t)$ satisfies the conditions of  Theorem \ref{the314}, then $f$ can be recovered from its  all above different  types of QFTs.

\end{remark}

\subsection{Problem  A for  QLCTs}
 In this subsection,  the inversion theorem of   two-sided QLCT  is studied.
\begin{theorem}($\textbf{ Inversion Theorem for two-sided QLCT}$ )\label{the317}\\
Suppose  in the cross-neighborhood  of $ (x_{0},y_{0})$, $f$ is  QBVF, $e^{\i\frac{a_{1}}{2b_{1}}s^{2}}f(s,t)e^{\j\frac{a_{2}}{2b_{2}}t^{2}}$belongs to $\bf{LC}$   and $f\in L^{1}(\mathbb{R}^2,\mathbb{H}) $,
then
\begin{eqnarray}\label{h82}
f(x_{0},y_{0})= \lim_{\begin{subarray}

N \to \infty \\M \to \infty
\end{subarray}
 }\frac{1}{4\pi^2}\int_{-N}^{N}\int_{-M}^{M}K_{A_{1}^{-1}}^{\i}(u,x_{0})\mathcal{L}_{T}^{\i,\j}(f)(u, v)K_{A_{2}^{-1}}^{\j}(v,y_{0})dudv,
 \end{eqnarray}

where
\begin{eqnarray*}
f(x_{0},y_{0})=\frac {f(x_{0}+0,y_{0}+0)+f(x_{0}+0,y_{0}-0)+f(x_{0}-0,y_{0}+0)+f(x_{0}-0,y_{0}-0)}{ 4}.
\end{eqnarray*}

\end{theorem}

\begin{proof}
Set
\begin{eqnarray*}
J(x_{0},y_{0},N,M)=\int_{-N}^{N}\int_{-M}^{M}K_{A_{1}^{-1}}^{\i}(u,x_{0})\mathcal{L}_{T}^{\i,\j}(f)(u, v)K_{A_{2}^{-1}}^{\j}(v,y_{0})dudv \label{hu38}
\end{eqnarray*}

and rewrite this expression by inserting the definition of $\mathcal{L}_{T}^{\i,\j}(f)(u, v)$ in the Equation (\ref{hu26}):
\begin{eqnarray}
&&J(x_{0},y_{0},N,M)=\int_{-N}^{N}\int_{-M}^{M}K_{A_{1}^{-1}}^{\i}(u,x_{0})\left(\int_{\mathbb{R}^2} K_{A_{1}}^{\i}(s,u)f(s,t)K_{A_{2}}^{\j}(t,v)dsdt)K_{A_{2}^{-1}}^{\j}(v,y_{0}\right)dudv  \nonumber \\
=&&\int_{-N}^{N}\int_{-M}^{M}\frac{1}{4\pi^2b_{1}b_{2}}\int_{\mathbb{R}^2} e^{-\i\frac{s-x_{0}}{b_{1}}u}e^{\i\frac{a_{1}(s^2-x_{0}^2)}{2b_{1}}}f(s,t)
 e^{-\j\frac{t-y_{0}}{b_{2}}v}e^{\j\frac{a_{2}(t^2-y_{0}^2)}{2b_{2}}}dsdtdudv \nonumber \\
=&&\int_{\mathbb{R}^2}\frac{1}{\pi^2} \frac{\sin(\frac{s-x_{0}}{b_{1}}N)}{s-x_{0}}e^{\i\frac{a_{1}(s^2-x_{0}^2)}{2b_{1}}}f(s,t)\frac{\sin(\frac{t-y_{0}}{b_{2}}M)}{t-y_{0}}
e^{\j\frac{a_{2}(t^2-y_{0}^2)}{2b_{2}}}dsdt \nonumber \\
=&&\int_{\mathbb{R}^2}\frac{1}{\pi^2} \frac{\sin(\frac{s-x_{0}}{b_{1}}N)}{s-x_{0}}e^{\i\frac{a_{1}(-x_{0}^2)}{2b_{1}}}g(s,t)\frac{\sin(\frac{t-y_{0}}{b_{2}}M)}{t-y_{0}}
e^{\j\frac{a_{2}(-y_{0}^2)}{2b_{2}}}dsdt, \nonumber \label{hu39}
\end{eqnarray}
where $g(s,t)=e^{\i\frac{a_{1}}{2b_{1}}s^{2}}f(s,t)e^{\j\frac{a_{2}}{2b_{2}}t^{2}}$.

Since $e^{\i\frac{a_{1}}{2b_{1}}s^{2}},$ $e^{\j\frac{a_{2}}{2b_{2}}t^{2}}$ are the $QBVFs$ in the rectangle-neighborhood of $(x_{0},y_{0})$, then $g(s,t)$ is also  a  $QBVF$ by using  Remark \ref{re32}.
 By Theorem \ref{the314}, we have
\begin{eqnarray*}
\lim_{\begin{subarray}

N \to \infty \\M \to \infty
\end{subarray}
 }e^{\i\frac{a_{1}}{2b_{1}}x_{0}^{2}}J(x_{0},y_{0},N,M)e^{\j\frac{a_{2}}{2b_{2}}y_{0}^{2}}
&&=\lim_{\begin{subarray}

N \to \infty \\M \to \infty
\end{subarray}
 }\int_{\mathbb{R}^2}\frac{1}{\pi^2} \frac{\sin((s-x_{0})\frac{N}{b_{1}})}{s-x_{0}}g(s,t)\frac{\sin((t-y_{0})\frac{M}{b_{2}})}{t-y_{0}}dsdt\nonumber\\
 &&=g(x_{0},y_{0}).
\end{eqnarray*}
That is to say:
 \begin{eqnarray*}
\lim_{\begin{subarray}

N \to \infty \\M \to \infty
\end{subarray}
 }J(x_{0},y_{0},N,M)
&&=\lim_{\begin{subarray}

N \to \infty \\M \to \infty
\end{subarray}
 }\int_{\mathbb{R}^2}\frac{1}{\pi^2} \frac{\sin((s-x_{0})\frac{N}{b_{1}})}{s-x_{0}}g(s,t)\frac{\sin((t-y_{0})\frac{M}{b_{2}})}{t-y_{0}}dsdt\\
 &&=e^{\i\frac{-a_{1}}{2b_{1}}x_{0}^{2}}g(x_{0},y_{0})e^{-\j\frac{a_{2}}{2b_{2}}y_{0}^{2}}\\
 &&=f(x_{0},y_{0}),
\end{eqnarray*}
this complete the proof.

\end{proof}

\begin{remark}\label{re318}
\begin{enumerate}
  \item If  $f\in L^{1}(\mathbb{R}^{2},\mathbb{ H}),$ is a QBVF in the rectangle-neighborhood $(x_{0}, y_{0})$, and its four component are separable i.e., $f_{n}(s,t)=f_{n}^{(1)}(s)f_{n}^{(2)}(t), n=0,1,2,3,$
 then the inversion formulae $ (\ref{h81})$, $(\ref{r1})$, $(\ref{l1})$ and $ (\ref{h82})$ can be   holden  without the condition $ \mathbf{LC}$.
  \item The proof of Theorem  \ref{the317} only works for the  two-sided QLCT, but not for the right-sided and left-sided QLCTs.  A straightforward computation shows that
\begin{eqnarray*}
&&\int_{-N}^{N}\int_{-M}^{M}\left(\int_{\mathbb{R}^2} f(s,t)K_{A_{1}}^{\bf i}(s,u)K_{A_{2}}^{\bf j}(t,v)dsdt\right)K_{A_{2}^{-1}}^{\j}(v,y_{0})K_{A_{1}^{-1}}^{\i}(u,x_{0})dudv\\
=&& \int_{-N}^{N}\int_{-M}^{M}\left(\int_{\mathbb{R}^2} \frac{1}{4\pi^2b_{1}b_{2}}f(s,t)e^{\i(\frac{a_{1}}{2b_{1}}s^2-\frac{1}{b_{1}}su+\frac{d_{1}}{2b_{1}}u^2)}
 e^{\j(\frac{a_{2}}{2b_{2}}t^2-\frac{1}{b_{2}}tv+\frac{d_{2}}{2b_{2}}v^2)}dsdt\right)\\
 ~&& e^{\j(\frac{-d_{2}}{2b_{2}}v^2+\frac{1}{b_{2}}y_{0}v-\frac{a_{2}}{2b_{2}}y_{0}^2)}e^{\i(\frac{-d_{1}}{2b_{1}}u^2+\frac{1}{b_{1}}x_{0}u-\frac{a_{1}}{2b_{1}}x_{0}^2)}dudv\\
= &&\int_{\mathbb{R}^2}f(s,t)\int_{-N}^{N}\frac{1}{4\pi^2b_{1}b_{2}}e^{\i(\frac{a_{1}}{2b_{1}}s^2-\frac{1}{b_{1}}su+\frac{d_{1}}{2b_{1}}u^2)}du\\
&&\int_{-M}^{M}e^{\j(\frac{1}{b_{2}}(y_{0}-t))v}dve^{\j(\frac{a_{2}}{2b_{2}}(t^{2}-y_{0}^{2}))}e^{\i(\frac{-d_{1}}{2b_{1}}u^2+\frac{1}{b_{1}}x_{0}u-\frac{a_{1}}{2b_{1}}x_{0}^2)}dsdt\\
=&&\int_{\mathbb{R}^2}f(s,t)\int_{-N}^{N}\frac{1}{4\pi^2b_{1}b_{2}}e^{\i(\frac{a_{1}}{2b_{1}}s^2-\frac{1}{b_{1}}su+\frac{d_{1}}{2b_{1}}u^2)}du\\
&&\frac{2\sin(\frac{1}{b_{2}}(y_{0}-t)M)}{\frac{1}{b_{2}}(y_{0}-t)}e^{\j(\frac{a_{2}}{2b_{2}}(t^{2}-y_{0}^{2}))}e^{\i(\frac{-d_{1}}{2b_{1}}u^2+\frac{1}{b_{1}}x_{0}u-\frac{a_{1}}{2b_{1}}x_{0}^2)}dsdt.
\end{eqnarray*}
The non-commutativity  of  $e^{\j(\frac{a_{2}}{2b_{2}}(t^{2}-y_{0}^{2}))} $  and $e^{\i(\frac{-d_{1}}{2b_{1}}u^2+\frac{1}{b_{1}}x_{0}u-\frac{a_{1}}{2b_{1}}x_{0}^2)}$ in the last equation gives the reason why  this method  fails to  right-sided and left-sided QLCTs, but in the following subsection, we show that the inversion formulas of  right-sided and left-sided  QLCTs can  be holden  pointwise almost everywhere  by making  use  of  the relations  between QFTs and QLCTs.
\end{enumerate}

\end{remark}

\subsection{Problem B for QFTs}
In this subsection, we first give the  sufficient conditions to solve the inversion problems of  different types of QFTs  in  $L^{1}(\mathbb{R}^{2}, \mathbb{H})$.
Moreover,  with the relations between two-sided QFT and two-sided QLCT, two-sided QLCT and  right-sided, left-sided QLCTs, we   obtain the  inversion problems of different types of QLCTs.
We first give two technical Lemmas.
\begin{lemma}\label{l51}
Suppose $f\in L^{1}(\mathbb{R}^{2}, \mathbb{H})$,  $w$ is the Gaussian function on $ \mathbb{R}^{2}$,  i.e., for $\alpha >0$
$w(x,y)=\frac{1}{4\pi^{2}}e^{-\alpha(x^{2}+y^{2})}$,then
\begin{eqnarray*}\label{F5}
\int_{\mathbb{R}^{2}}\mathcal{F}_{T}(x,y)w(x,y)dxdy=\int_{\mathbb{R}^{2}}f(s,t)\mathcal{W}_{T}(s,t)dsdt.
\end{eqnarray*}
\end{lemma}
\begin{proof}
Since
\begin{eqnarray*}
\mathcal{W}_{T}(s,t)&&=\frac{1}{4\pi^{2}}\int_{\mathbb{R}^{2}}e^{-\i sx}e^{-\alpha(x^{2}+y^{2})}e^{-\j ty}dxdy\\
&&= \frac{1}{4\pi^{2}}\int_{\mathbb{R}}e^{-\i sx}e^{-\alpha x^{2}}dx\int_{\mathbb{R}}e^{-\alpha y^{2}}e^{-\j ty}dy=\frac{1}{4\pi \alpha}e^{-\frac{s^{2}+t^{2}}{4\alpha}},
\end{eqnarray*}
then $\mathcal{W}_{T}(s,t)$ is the Gauss-Weierstrass kernels in $\mathbb{R}^{2}$.
\par
Since
\begin{eqnarray*}
\left |\int_{\mathbb{R}^{2}}\mathcal{F}_{T}(x,y)w(x,y)dxdy \right |
\leq \int_{\mathbb{R}^{2}}\int_{\mathbb{R}^{2}}\left |e^{-\i sx}f(s,t)e^{-\j ty}w(x,y)\right |dsdtdxdy
\leq  \|f\|_{1}\|w\|_{1},
\end{eqnarray*}
then $\int_{\mathbb{R}^{2}}\mathcal{F}_{T}(x,y)w(x,y)dxdy $ is well defined, and by the Fubini Theorem, we have that
\begin{eqnarray*}
\int_{\mathbb{R}^{2}}\mathcal{F}_{T}(x,y)w(x,y)dxdy
&&=\frac{1}{4\pi^{2}}\int_{\mathbb{R}^{2}}\int_{\mathbb{R}^{2}}e^{-\i sx}f(s,t)e^{-\j ty} e^{-\alpha(x^{2}+y^{2})}dsdtdxdy\\
&&=\frac{1}{4\pi^{2}}\int_{\mathbb{R}^{2}}\int_{\mathbb{R}^{2}}e^{-\i sx}f(s,t)e^{-\j ty} e^{-\alpha(x^{2}+y^{2})}dsdtdxdy\\
&&=\frac{1}{4\pi^{2}}\int_{\mathbb{R}}e^{-\alpha x^{2}}e^{-\i sx}dx\int_{\mathbb{R}^{2}}f(s,t)dsdt\int_{\mathbb{R}}e^{-\j ty} e^{-\alpha y^{2}}dy\\
&&=\frac{1}{4\pi \alpha}\int_{\mathbb{R}^{2}}f(s,t)e^{-\frac{s^{2}+t^{2}}{4a}} dsdt =\int_{\mathbb{R}^{2}}f(s,t)\mathcal{W}_{T}(s,t)dsdt.
\end{eqnarray*}
\end{proof}

From Lemma \ref{l51}, if $ f\in L^{1}(\mathbb{R}^{2}, \mathbb{H}) $, then we have

\begin{eqnarray}
&&\frac{1}{4\pi^{2}}\int_{\mathbb{R}^{2}}e^{\i sx}\mathcal{F}_{T}(x,y)e^{\j ty}e^{-\alpha(x^{2}+y^{2})} dxdy \nonumber \\
=&&\frac{1}{4\pi^{2}}\int_{\mathbb{R}^{2}}e^{\i sx}\int_{\mathbb{R}^{2}}e^{-\i xu}f(u,v)e^{-\j yv}dudve^{\j ty}e^{-\alpha(x^{2}+y^{2})}dxdy \nonumber\\
=&&\frac{1}{4\pi \alpha}\int_{\mathbb{R}^{2}}f(u,v)e^{-\frac{(u-s)^{2}+(v-t)^{2}}{4a}}dudv \nonumber\\
=&&f_{0}\ast \mathcal{W}_{T}+ f_{1}\ast \mathcal{W}_{T} \i+f_{2}\ast \mathcal{W}_{T} \j+f_{3}\ast \mathcal{W}_{T}\k.
\end{eqnarray}

\begin{lemma}\label{l52}
If $f \in L^{1}(\mathbb{R}^{2}, \mathbb{H})$, $\mathcal{W}_{T}(s,t)$ is the Gauss-Weierstrass kernel in  $\mathbb{R}^{2}$,
$\mathcal{W}_{T}(s,t)=\frac{1}{4\pi \alpha}e^{-\frac{s^{2}+t^{2}}{4\alpha}}$, then
\begin{eqnarray*}
\lim_{\begin{subarray}
\alpha \alpha \to 0^{+}
\end{subarray}} \| f\ast \mathcal{W}_{T}-f \|_{1}=0.
\end{eqnarray*}
\end{lemma}
\begin{proof}
By Theorem 1.18 in \cite{stein1971introduction}, if $ f_{n}\in L^{1}(\mathbb{R}^{2},\mathbb{R}),$ then
$$\|f_{n}\ast \mathcal{W}_{T}-f_{n} \|_{1}\rightarrow 0,\quad as \quad \alpha \rightarrow 0, n=0,1,2,3.$$

Since $f \in L^{1}(\mathbb{R}^{2}, \mathbb{H})$, using the Lemma \ref{le313}, the components $f_{n}$ of $f$ are all $ \in L^{1}(\mathbb{R}^{2}, \mathbb{R}), n=0,1,2,3,$
and
$$\| f\ast \mathcal{W}_{T}(s,t)-f\|_{1}\leq \sum_{n=0}^{3} \| f_{n}\ast \mathcal{W}_{T}(s,t)-f_{n}\|_{1},$$
then  $$\| f\ast \mathcal{W}_{T}(s,t)-f\|_{1}\rightarrow 0,\quad as \quad \alpha \rightarrow 0, $$
that is to say, the Gauss means of the integral $\frac{1}{4\pi^{2}}\int_{\mathbb{R}^{2}}e^{\i sx}\mathcal{F}_{T}(x,y)e^{\j ty} dxdy $ converge to $f(s,t)$ in the $L^{1}$ norm.

\end{proof}

We are now ready to prove one of our main results.
\begin{theorem}($\textbf{ Inversion Theorem for two-sided QFT}$)\label{T53}\\
Suppose $f$ and $\mathcal{F}_{T} \in L^{1}(\mathbb{R}^{2}, \mathbb{H})$, then
\begin{eqnarray}\label{I3}
f(s,t)=\frac{1}{4\pi^{2}}\int_{\mathbb{R}^{2}}e^{\i su}\mathcal{F}_{T}(u,v)e^{\j tv} dudv,
\end{eqnarray}
for almost everywhere $(s,t)$.
\end{theorem}
\begin{proof}
From  Lemma \ref{l52},
since  $$ \left \| \frac{1}{4\pi^{2}}\int_{\mathbb{R}^{2}}e^{\i su}\mathcal{F}_{T}(u,v)e^{\j ty}e^{-\alpha(u^{2}+v^{2})} dudv-f(s,t) \right \|_{1}\longrightarrow 0, \quad as \quad \alpha \rightarrow 0,$$
there exists a sequence $ \alpha_{k}\longrightarrow 0 $ such that $\frac{1}{4\pi^{2}}\int_{\mathbb{R}^{2}}e^{\i su}\mathcal{F}_{T}(u,v)e^{\j tv}e^{-\alpha(u^{2}+v^{2})} dudv\longrightarrow f(s,t) $ for almost everywhere $ (s,t)$.

\begin{eqnarray*}\label{F4}
f(s,t)=\lim_{\begin{subarray}
\alpha \alpha_{k} \to 0^{+}
\end{subarray}}\frac{1}{4\pi^{2}}\int_{\mathbb{R}^{2}}e^{\i su}\mathcal{F}_{T}(u,v)e^{\j tv}e^{-\alpha_{k}(u^{2}+v^{2})} dudv.
\end{eqnarray*}
Since $\mathcal{F}_{T}\in L^{1}(\mathbb{R}^{2}, \mathbb{H})$, the quaternion Lebesgue dominated convergence theorem \cite{kou2016envelope} gives us the following pointwise equality
\begin{eqnarray*}
f(s,t)=\frac{1}{4\pi^{2}}\int_{\mathbb{R}^{2}}e^{\i sx}\mathcal{F}_{T}(x,y)e^{\j ty} dsdt, a.e..
\end{eqnarray*}
\end{proof}

With the similar argument, we can obtain the inversion theorem of right-sided and left-sided QFTs.

\begin{theorem}($\textbf{ Inversion Theorem for right-sided QFT}$)\label{R1}\\
Suppose $f $ and $\mathcal{F}_{R} \in L^{1}(\mathbb{R}^{2}, \mathbb{H})$, then
\begin{eqnarray}\label{RR1}
f(s,t)=\frac{1}{4\pi^{2}}\int_{\mathbb{R}^{2}}\mathcal{F}_{R}(u,v)e^{\j tv}e^{\i su} dudv,
\end{eqnarray}
for almost everywhere $(s,t)$.
\end{theorem}
\begin{theorem}($\textbf{ Inversion Theorem for left-sided QFT}$)\label{L1}\\
Suppose $f$  and $\mathcal{F}_{L} \in L^{1}(\mathbb{R}^{2}, \mathbb{H})$, then
\begin{eqnarray}\label{LL1}
f(s,t)=\frac{1}{4\pi^{2}}\int_{\mathbb{R}^{2}}e^{\j tv}e^{\i su}\mathcal{F}_{L}(u,v) dudv,
\end{eqnarray}
for almost everywhere $(s,t)$.
\end{theorem}
In what follows, another sufficient conditions for the inversion formulas of  QFTs  hold pointwisly in quaternion field is below.
\begin{coro}\label{T55}
Suppose $f\in L^{1}(\mathbb{R}^2,\mathbb{H}),$  then $ f$ can be restructured by its two-sided QFT function  as in
Equation $(\ref{I3})$  if one of the following conditions hold.
\begin{description}
\item[$(I).$] $\mathcal{F}_{T,n} \in L^{1}(\mathbb{R}^2,\mathbb{H}).$
  \item[$(II).$] $F_{n}\in L^{1}(\mathbb{R}^2, \mathbb{C}).$
  \item[$(III).$] $f$ is continuous at $(0,0)$, $ F_{n}\geq 0 .$
\end{description}
where
$ \mathcal{F}_{T,n} $ and  $F_{n} $ are the two-sided QFT and the 2D FT of $f_{n}, n=0,1,2,3$, respectively, which are the components of the $f$. 2D FT  is defined in Equation $(\ref{e9}).$

\end{coro}

\begin{proof}
For conditions $(I),$   since $$ \| \mathcal{F}_{T}\|_{1} \leq \sum_{0}^{3}\| \mathcal{F}_{T,n}\|_{1} < \infty ,$$
therefore, from  Theorem \ref{T53}, Equation $(\ref{I3})$ holds .
\par
For conditions $(II),$
The relationship between two-sided QFT $ \mathcal{H}_{T} $ and 2D FT $ H $ of  a real integrable  function $h$ is given as follows:
\begin{eqnarray}\label{F1}
\mathcal{H}_{T}(u,v)=\frac{ H(u,v)(1-\k)+H(u,-v)(1+\k) }{2},
\end{eqnarray}
\begin{eqnarray}\label{F2}
H(u,v)=\frac{ \mathcal{H}_{T}(u,v)(1+\k)+\mathcal{H}_{T}(u,-v)(1-\k) }{2},
\end{eqnarray}
 Equation (\ref{F1}) was given in \cite{pei2001efficient}, while  Equation (\ref{F2}) also can be proved by similar argument, we omit it.
By Equations (\ref{F1}) and (\ref{F2}),  $\mathcal{F}_{T,n} \in L^{1}(\mathbb{R}^2,\mathbb{H}) $ if and only if  $F_{n} \in L^{1}(\mathbb{R}^2,\mathbb{C}). $ Therefore statement $(I)$ and $ (II)$ are equivalent.
\par
For statement $(III),$
since $f$ is continuous at $(0,0)$ and $ F_{n}\geq 0 $, then $F_{n}\in L^{1}(\mathbb{R}^2, \mathbb{C})$ due to  \cite{stein1971introduction}, hence, from the statement $(II) ,$ Equation $(\ref{I3})$ holds .
\end{proof}

\begin{remark}\label{rem319}
\begin{enumerate}
  \item We can replace the two-sided QFT by right-sided, or left-sided QFTs in the Equations $(\ref{F1})$ and $(\ref{F2})$.
  \item It is easy to show that The  Corollary $\ref{T55}$    not only  work for two-sided QFT, but also work for the right-sided and left-sided QFTs and the generalized QFTs in Remark $\ref{re316}$.
\end{enumerate}

\end{remark}
Before giving  other sufficient conditions of inversion theorem for QFTs, we introduce the following concept \cite{stein1971introduction}:
\begin{defn}
 If $f\in L^{1}(\mathbb{R}^{2}, \mathbb{H})$ is differentiable  in the $L^{1}$ norm with respect to $s$
and there exists a function $ g \in L^{1}(\mathbb{R}^{2}, \mathbb{H})$ such that
$$\lim_{\begin{subarray}
\\h \to 0
\end{subarray}} \int_{\mathbb{R}^{2}} \left | \frac{f(s+h,t)-f(s,t)}{h} -g(s,t) \right | dsdt=0,$$
then the function $g$ is said to be \textbf{ the partial derivative of $f$ with respect to $s$ in the $L^{1}$ norm}.
\end{defn}
\begin{lemma}\label{h91}
If $f\in L^{1}(\mathbb{R}^{2}, \mathbb{H})$, and $g(s,t)$ is the partial derivative of $f(s,t)$ with respect to $s$ in the $L^{1}$ norm then
\begin{eqnarray*}
\mathcal{G}_{T}(u,v)=\i u\mathcal{F}_{T}(u,v).
\end{eqnarray*}
\end{lemma}

\begin{proof}
Since  $g(s,t)$ is the partial derivative of $f$ with respect to $s$ in the $L^{1}$ norm, then
$$\lim_{\begin{subarray}
\\h \to 0
\end{subarray}} \int_{\mathbb{R}^{2}} \left | \frac{f(s+h,t)-f(s,t)}{h} -g(s,t) \right | dsdt=0,$$
it's easy to see that the two-sided QFT of $ \frac{f(s+h,t)-f(s,t)}{h} -g(s,t)$ is $  \frac{e^{\i uh}\mathcal{F}_{T}(u,v)-\mathcal{G}_{T}(u,v)}{h} $, and
$$
\left |\frac{e^{\i uh}\mathcal{F}_{T}(u,v)-\mathcal{F}_{T}(u,v)}{h}-\mathcal{G}_{T}(u,v) \right |  \leq \int_{\mathbb{R}^{2}} \left | \frac{f(s+h,t)-f(s,t)}{h} -g(s,t) \right| dsdt,
$$
then letting $ h \to 0,$ we obatin:

\begin{eqnarray*}
\mathcal{G}_{T}(u,v)=\i u\mathcal{F}_{T}(u,v).
\end{eqnarray*}
\end{proof}

The Lemma \ref{h91} can be extended to higher derivatives by induction  as follows:
\begin{theorem}\label{h92}
If $f \in L^{1}(\mathbb{R}^{2}, \mathbb{H})$ has derivatives in the $ L^{1}(\mathbb{R}^{2}, \mathbb{H})$ norm of all orders $ \leq m+n$, then
\begin{eqnarray} \label{p92}
\mathcal{F}_{T}\left(\frac{\partial ^{m+n}}{\partial s^{m}\partial t^{n}}f\right ) (u,v)=(\i u)^{m}\mathcal{F}_{T}(u,v)(\j v)^{n},
\end{eqnarray}
where $\mathcal{F}_{T}(\frac{\partial ^{m+n}}{\partial s^{m}\partial t^{n}}f) (u,v)$ is the two-sided QFT of $\frac{\partial ^{m+n}}{\partial s^{m}\partial t^{n}}f(s,t). $
\end{theorem}
By the non - commutativity of quaternions, we obtain the following results for left-sided and right-sided QFTs.

\begin{theorem}
If $f(s,t) \in L^{1}(\mathbb{R}^{2}, \mathbb{H})$ has derivatives  with respect to $s$ in the $ L^{1}(\mathbb{R}^{2}, \mathbb{H})$ norm of all orders $ \leq m$, then
\begin{eqnarray*} \label{p93}
\mathcal{F}_{L}\left(\frac{\partial ^{m}}{\partial s^{m}}f\right) (u,v)=(\i u)^{m}\mathcal{F}_{L}(u,v),
\end{eqnarray*}
where $\mathcal{F}_{L}(\frac{\partial ^{m}}{\partial s^{m}}f) (u,v)$ is the left-sided QFT of $\frac{\partial ^{m}}{\partial s^{m}}f(s,t) .$
\end{theorem}
\begin{theorem}
If $f(s,t) \in L^{1}(\mathbb{R}^{2}, \mathbb{H})$ has derivatives  with respect to $t$ in the $ L^{1}(\mathbb{R}^{2}, \mathbb{H})$ norm of all orders $ \leq n$, then
\begin{eqnarray*} \label{p94}
\mathcal{F}_{R}\left(\frac{\partial ^{n}}{\partial t^{n}}f\right) (u,v)=\mathcal{F}_{R}(u,v)(\j v)^{n},
\end{eqnarray*}
where $\mathcal{F}_{R}(\frac{\partial ^{n}}{\partial t^{n}}f) (u,v)$ is the right-sided QFT of $\frac{\partial ^{n}}{\partial t^{n}}f(s,t). $
\end{theorem}

\begin{theorem}($\textbf{ Inversion Theorem for two-sided QFT}$)\label{h93}\\
If $f \in L^{1}(\mathbb{R}^{2}, \mathbb{H})$ has derivatives in the $ L^{1}(\mathbb{R}^{2}, \mathbb{H})$ norm of all orders $ \leq 3$, then

\begin{eqnarray*}
f(s,t)=\frac{1}{4\pi^{2}}\int_{\mathbb{R}^{2}}e^{\i su}\mathcal{F}_{T}(u,v)e^{\j tv} dudv,
\end{eqnarray*}
for almost everywhere $(s,t)$.
\end{theorem}
\begin{proof}
Let $z=(u,v)\in \mathbb{R}^{2}$, then there exists  constant $C$ such that
\begin{eqnarray*}
(1+|z|^{2})^{\frac{3}{2}} \leq (1+|u|^{2}+|v|^{2})^{3} \leq C \sum_{|\alpha|\leq 3}|z^{\alpha}|,
\end{eqnarray*}
where $z^{\alpha}=u^{m}v^{n},|\alpha|=m+n $.
then by  Equation $(\ref{p92})$,
\begin{eqnarray*}
|\mathcal{F}_{T}(u,v)|&&\leq (1+|z|^{2})^{-\frac{3}{2}} C\sum_{|\alpha|\leq 3}|z^{\alpha}||\mathcal{F}_{T}(u,v)|\\
&&=(1+|z|^{2})^{-\frac{3}{2}} C\sum_{|\alpha|\leq 3}\left|\mathcal{F}_{T}\left(\frac{\partial ^{m+n}}{\partial s^{m}\partial t^{n}}f\right)(u,v)\right|\\
&&\leq (1+|z|^{2})^{-\frac{3}{2}}C \sum_{|\alpha|\leq 3}\left \|\frac{\partial ^{m+n}}{\partial s^{m}\partial t^{n}}f\right\|_{1}.
\end{eqnarray*}
Since $ (1+|z|^{2})^{-\frac{3}{2}} $ is an integrable function on $\mathbb{R}^{2} $, it follow that $\mathcal{F}_{T} \in L^{1}(\mathbb{R}^{2}, \mathbb{H})$, then by Theorem $ \ref{T53}$, the proof is completed.

\end{proof}
\begin{remark}
The above method can only be applied to  two-sided QFT, but not for left-sided and right-sided QFTs, because of  non - commutativity of quaternions.
\end{remark}
\subsection{Problem B for QLCTs}
In  this subsection,   by using the relationship between QFTs and QLCTs, we firstly prove the following Lemma,
which  arises from  the relationship  between two-sided QLCT and QFT \cite{bahri2014relationship}, then we derive the inversion of QLCTs.
\begin{lemma} \label{l55}

$$f(s,t)\in L(\mathbb{R}^2,\mathbb{H}) \quad  if \quad and \quad only \quad if  \quad  p(s,t)\in L(\mathbb{R}^2,\mathbb{H})$$ and
 $$\mathcal{L}_{T}^{\i,\j}(f)(u,v )\in L(\mathbb{R}^2,\mathbb{H})  \quad  if \quad and \quad only \quad if  \quad   \quad \mathcal{P}_{T}(u,v)\in L(\mathbb{R}^2,\mathbb{H}),$$
where

 \begin{eqnarray*}
p(s,t):=e^{\i\frac{a_{1}}{2b_{1}}s^{2}}f(s,t)e^{\j\frac{a_{2}}{2b_{2}}t^{2}}.\label{hu47}
 \end{eqnarray*}
 and $\mathcal{P}_{T}(u,v)$ is the two-sided QFT of  $p(s,t).$
 \end{lemma}

\begin{proof}
Using
\begin{eqnarray*}\label{QFR}
 \mathcal{L}_{T}^{\i,\j}(f)(u,v )&&=\int_{\mathbb{R}^2} K_{A_{1}}^{\i}(s,u)f(s,t)K_{A_{2}}^{\j}(t,v)dsdt \nonumber\\
 &&=\frac{1}{\sqrt{\i2b_{1}\pi}}e^{\i\frac{d_{1}}{2b_{1}}u^{2}} \mathcal{P}_{T}(\frac{1}{b_{1}}u,\frac{1}{b_{2}}v)\frac{1}{\sqrt{\j2b_{2}\pi}}e^{\j\frac{d_{2}}{2b_{2}}v^{2}}.\label{hu46}
 \end{eqnarray*}
\end{proof}

By Lemma \ref{l55} and Theorem \ref{T53}, the inversion theorem of  two-sided QLCT are presented.

\begin{theorem}($\textbf{ Inversion Theorem for two-sided QLCT}$)\label{T57}\\
If one of the following conditions hold,
\begin{description}
  \item[$(\alpha)$]  $f $ and $\mathcal{L}_{T}^{\i,\j}(f) \in L^{1}(\mathbb{R}^{2}, \mathbb{H})$,
  \item[$(\beta)$] For $e^{\i\frac{a_{1}}{2b_{1}}s^{2}}f(s,t)e^{\j\frac{a_{2}}{2b_{2}}t^{2}}\in L^{1}(\mathbb{R}^{2}, \mathbb{H})$ has derivatives in the $ L^{1}(\mathbb{R}^{2}, \mathbb{H})$ norm of all orders $ \leq 3$,
\end{description}
then the original function $f$ can be recovered from its
two-sided QLCT  by Equation $(\ref{F3})$
\begin{eqnarray}\label{F3}
f(s,t)=\mathcal{L}_{T^{-1}}^{\i,\j}(     \mathcal{L}_{T}^{\i,\j}(f))(s,t )=\int_{\mathbb{R}^{2}}K_{A_{1}^{-1}}^{\i}(u,s)\mathcal{L}_{T}^{\i,\j}(f)(u,v )K_{A_{2}^{-1}}^{\j}(v,t)dudv,
\end{eqnarray}
for almost everywhere $(s,t)$.
\end{theorem}
\begin{proof}
In one hand, if $f$ satisfies conditions $ (\alpha),$
Lemma \ref{l55} follows that if $f$ and $\mathcal{L}_{T}^{\i,\j}(f) \in L^{1}(\mathbb{R}^{2}, \mathbb{H})$, then $ p(s,t)$ and its two-sided QFT $ \mathcal{P}_{T}(u,v)$  also belong to $L^{1}(\mathbb{R}^{2}, \mathbb{H})$,
by Theorem $ \ref{T53}$, the $ p(s,t)$ can be recovered from its QFT almost everywhere as follows:
\begin{eqnarray*}
p(s,t)=\frac{1}{4 \pi^{2}}\int_{\mathbb{R}^{2}}e^{\i us}\mathcal{P}_{T}(u,v)e^{\j vt}dudv, a.e.,
\end{eqnarray*}
then, from Lemma \ref{l55}, a straightforward calculation gives:
\begin{eqnarray*}
e^{\i\frac{a_{1}}{2b_{1}}s^{2}}f(s,t)e^{\j\frac{a_{2}}{2b_{2}}t^{2}}=&&\frac{1}{4b_{1}b_{2} \pi^{2}}
\int_{\mathbb{R}^{2}}e^{\i \frac{1}{b_{1}}us}\mathcal{P}_{T}(\frac{1}{b_{1}}u,\frac{1}{b_{2}}v)e^{\j \frac{1}{b_{2}}vt}dudv,
\\
=&&b_{1}b_{2}\frac{1}{4 \pi^{2}}
\int_{\mathbb{R}^{2}}e^{\i \frac{1}{b_{1}}us}\sqrt{\i2b_{1}\pi} e^{-\i\frac{d_{1}}{2b_{1}}u^{2}} \mathcal{L}_{T}^{\i,\j}(f)(u,v )\sqrt{\j2b_{2}\pi}e^{-\i\frac{d_{2}}{2b_{2}}v^{2}} e^{\j \frac{1}{b_{2}}vt}dudv,\\
f(s,t)=&&
\int_{\mathbb{R}^{2}}\frac{1}{\sqrt{-\i2b_{1}\pi} }e^{-\i(\frac{a_{1}}{2b_{1}}s^{2}+\frac{1}{b_{1}}us-\frac{d_{1}}{2b_{1}}u^{2})} \mathcal{L}_{T}^{\i,\j}(f)(u,v )\frac{1}{\sqrt{-\j2b_{2}\pi} }e^{-\j(\frac{a_{2}}{2b_{2}}t^{2}+\frac{1}{b_{2}}vt-\frac{d_{2}}{2b_{2}}v^{2})}dudv,a.e.,
\end{eqnarray*}
the proof of the theorem for  conditions $ (\alpha)$ is  completed.
\par
On the other hand,
if $f$ satisfies conditions $ ( \beta),$ that is to say,
$p\in L^{1}(\mathbb{R}^{2}, \mathbb{H})$ has derivatives in the $ L^{1}(\mathbb{R}^{2}, \mathbb{H})$ norm of all orders $ \leq 3$, according to  Theorem \ref{h93}, it follows that $ \mathcal{P}_{T} \in  L^{1}(\mathbb{R}^{2}, \mathbb{H})$, then  Lemma \ref{l55} implies that $ \mathcal{L}_{T}^{\i,\j}(f) \in  L^{1}(\mathbb{R}^{2}, \mathbb{H})$.  Hence by conditions $ (\alpha),$ we can complete our proof for conditions $ (\beta)$.
\end{proof}

Using the  relationship  between  two-sided QLCT and right-sided, left-sided QLCTs, the existence and invertibility of right-sided and left-sided QLCTs can be inherited from the two-sided QLCT.
\begin{lemma}\label{relation}
If $f\in L^{1}(R^{2}, \mathbb{H})$, the right-sided and left-sided QLCTs  of $f$ can be decomposed into  the sum of two two-sided QLCT.

\begin{eqnarray} \label{f6}
 \mathcal{L}_{R}^{\i,\j}(f)(u,v)= \mathcal{L}_{T}^{\i,\j}(f_{a})(u,v)+ \mathcal{L}_{T}^{-\i,\j}(f_{b})(u,v)\j,
 \end{eqnarray}
\begin{eqnarray} \label{f7}
\mathcal{L}_{L}^{\i,\j}(f)(u,v)= \mathcal{L}_{T}^{\i,\j}(f_{d})(u,v)+ \i\mathcal{L}_{T}^{\i,-\j}(f_{e})(u,v),
 \end{eqnarray}
where
\begin{eqnarray}\label{f8}
f=f_{a}+f_{b}\j, f_{a}:=f_{0}+\i f_{1}, f_{b}:=f_{2}+\i f_{3}, f=f_{d}+\i f_{e}, f_{d}:=f_{0}+\j f_{2}, f_{e}:=f_{1}+\j f_{3}.
\end{eqnarray}
\end{lemma}

\begin{proof}
It suffices to prove  Equation (\ref{f6}), Equation (\ref{f7}) can be  proved  in the similar way, so we omit it .
\begin{eqnarray*}
 \mathcal{L}_{R}^{\i,\j}(f)(u,v)=&&\int_{\mathbb{R}^2}[ f_{a}(s,t)+f_{b}(s,t)\j]K_{A_{1}}^{ \i}(s,u)K_{A_{2}}^{ \j}(t,v)dxdy \nonumber \\
=&&\int_{\mathbb{R}^2}[ f_{a}(s,t)]K_{A_{1}}^{ \i}(s,u)K_{A_{2}}^{ \j}(t,v)dsdt
 +\int_{\mathbb{R}^2}[ f_{b}(s,t)]K_{A_{1}}^{ -\i}(s,u)K_{A_{2}}^{ \j}(t,v)dsdt \j \nonumber \\
 =&&\int_{\mathbb{R}^2}K_{A_{1}}^{ \i}(s,u)f_{a}(s,t)K_{A_{2}}^{ \j}(t,v)dsdt
 +\int_{\mathbb{R}^2}K_{A_{1}}^{ -\i}(s,u) f_{b}(s,t)K_{A_{2}}^{ \j}(t,v)dsdt \j, \nonumber \\
\end{eqnarray*}

\end{proof}


Therefore, using suitable conditions in Theorems \ref{the317}, \ref{T57} and  Lemma \ref{relation},
we have the following inversion   formulas  of right-sided and left-sided QLCTs.
\begin{coro}($\textbf{ Inversion formulas for right-sided and left-sided QLCTs}$)\label{c419}\\
Suppose   $f\in L^{1}(\mathbb{R}^2,\mathbb{H}) $,  in the cross-neighborhood  of  point $(x_{0},y_{0}),$ $f(s,t)$ is a  QBVF,
\\
if $e^{\i\frac{a_{1}}{2b_{1}}s^{2}}f_{a}(s,t)e^{\j\frac{a_{2}}{2b_{2}}t^{2}}$ and
  $e^{-\i\frac{a_{1}}{2b_{1}}s^{2}}f_{b}(s,t)e^{\j\frac{a_{2}}{2b_{2}}t^{2}}$ belongs to $\bf{LC},$
then the inversion formula of right-sided  QLCT of $f $ is obtained as follows:
\begin{eqnarray}\label{xx2}
f(x_{0},y_{0})&&= \lim_{\begin{subarray}
N N\to \infty \\M \to \infty
\end{subarray}
 }\int_{-N}^{N}\int_{-M}^{M}K_{A_{1}^{-1}}^{\i}(u,x_{0})\mathcal{L}_{T}^{\i,\j}(f_{a})(u,v)K_{A_{2}^{-1}}^{\j}(v,y_{0})dudv \nonumber\\
 &&+\lim_{\begin{subarray}
NN \to \infty \\M \to \infty
\end{subarray}
 }\int_{-N}^{N}\int_{-M}^{M}K_{A_{1}^{-1}}^{-\i}(u,x_{0})\mathcal{L}_{T}^{-\i,\j}(f_{b})(u,v)K_{A_{2}^{-1}}^{\j}(v,y_{0})dudv\j.
\end{eqnarray}
If
$e^{\i\frac{a_{1}}{2b_{1}}s^{2}}f_{d}(s,t)e^{\j\frac{a_{2}}{2b_{2}}t^{2}}$  and
  $e^{\i\frac{a_{1}}{2b_{1}}s^{2}}f_{e}(s,t)e^{-\j\frac{a_{2}}{2b_{2}}t^{2}}$ belongs to $\bf{LC},$\\
  then
the inversion formula of left-sided  QLCT of $f $ is obtained as follows:
\begin{eqnarray}\label{xxx2}
f(x_{0},y_{0})&&= \lim_{\begin{subarray}
N N\to \infty \\M \to \infty
\end{subarray}
 }\int_{-N}^{N}\int_{-M}^{M}K_{A_{1}^{-1}}^{\i}(u,x_{0})\mathcal{L}_{T}^{\i,\j}(f_{d})(u,v)K_{A_{2}^{-1}}^{\j}(v,y_{0})dudv \nonumber\\
 &&+\i\lim_{\begin{subarray}
NN \to \infty \\M \to \infty
\end{subarray}
 }\int_{-N}^{N}\int_{-M}^{M}K_{A_{1}^{-1}}^{\i}(u,x_{0}) \mathcal{L}_{T}^{\i,-\j}(f_{e})(u,v)K_{A_{2}^{-1}}^{-\j}(v,y_{0})dudv,
\end{eqnarray}
where
$$
f(x_{0},y_{0})=\frac {f(x_{0}+0,y_{0}+0)+f(x_{0}+0,y_{0}-0)+f(x_{0}-0,y_{0}+0)+f(x_{0}-0,y_{0}-0)}{ 4}.$$

\end{coro}

\begin{coro}($\textbf{ Inversion formula for right-sided QLCT}$)\label{c420}\\
Let $f  \in L^{1}(\mathbb{R}^{2}, \mathbb{H}) $. If
\par
$1)$ $\mathcal{L}_{T}^{\i,\j}(f_{a})$ and $\mathcal{L}_{T}^{-\i,\j}(f_{b})  \in L^{1}(\mathbb{R}^{2}, \mathbb{H}) $,
 \par
or
\par
$2) e^{\i\frac{a_{1}}{2b_{1}}s^{2}}f_{a}(s,t)e^{\j\frac{a_{2}}{2b_{2}}t^{2}}$ and
 $ e^{-\i\frac{a_{1}}{2b_{1}}s^{2}}f_{b}(s,t)e^{\j\frac{a_{2}}{2b_{2}}t^{2}}$ both have derivatives in the $ L^{1}(\mathbb{R}^{2}, \mathbb{H})$ norm of all orders $ \leq 3$,

then the inversion formula of right-sided QLCT of $f $ is
\begin{eqnarray}\label{rr1}
f(s,t)= \mathcal{L}_{T^{-1}}^{\i,\j}(\mathcal{L}_{T}^{\i,\j}(f_{a}))(s,t)+ \mathcal{L}_{T^{-1}}^{-\i,\j}(\mathcal{L}_{T}^{-\i,\j}(f_{b}))(s,t)\j. \quad a.e..
\end{eqnarray}
\end{coro}

\begin{coro}($\textbf{ Inversion formula for left-sided QLCT}$)\label{c421}\\
Let $f\in L^{1}(\mathbb{R}^{2}, \mathbb{H}).$ If
\par
$1) \mathcal{L}_{T}^{\i,\j}(f_{d})$ and $\mathcal{L}_{T}^{\i,-\j}(f_{e})\in L^{1}(\mathbb{R}^{2}, \mathbb{H}) $,
 \par
or
\par
$2)  e^{\i\frac{a_{1}}{2b_{1}}s^{2}}f_{d}(s,t)e^{\j\frac{a_{2}}{2b_{2}}t^{2}}$ and
 $ e^{\i\frac{a_{1}}{2b_{1}}s^{2}}f_{e}(s,t)e^{-\j\frac{a_{2}}{2b_{2}}t^{2}}$ both have  derivatives in the $ L^{1}(\mathbb{R}^{2}, \mathbb{H})$ norm of all orders $ \leq 3$,\\
then the inversion formula of left-sided QLCT of $f$ is
\begin{eqnarray}\label{ll1}
f(s,t)= \mathcal{L}_{T^{-1}}^{\i,\j}(\mathcal{L}_{T}^{\i,\j}(f_{d}))(s,t)+\i \mathcal{L}_{T^{-1}}^{\i,-\j}(\mathcal{L}_{T}^{\i,-\j}(f_{e}))(s,t), \quad  a.e..
\end{eqnarray}
\end{coro}

\begin{remark}
From the above theorem, we can find that the original function can be recovered from its right-sided and left-sided QLCTs by two-sided QLCTs of its components.
\end{remark}
The proof of following Lemma $\ref{l81}$ is straightforward.
 \begin{lemma}\label{l81}
 If $f\in L^{1}(\mathbb{R}^{2}, \mathbb{H})$, then
\begin{eqnarray*}
\mathcal{L}_{R}^{\i}(f)(u,t)
&=&\mathcal{F}_{r}^{\i}(f(\cdot,t) \frac{1}{\sqrt{2b_{1}\i\pi}}e^{\i\frac{a_{1}}{2b_{1}}(\cdot)^{2}})
(\frac{u}{b_{1}}, t)e^{\i\frac{d_{1}}{2b_{1}}u^{2}},\\
\mathcal{L}_{R}^{\i,\j}(f)(u,v)
&=&\mathcal{F}_{r}^{\j}\bigg(\mathcal{L}_{R}^{\i}(f)
(u, \cdot)e^{\j\frac{a_{2}}{2b_{2}}(\cdot)^{2}}\bigg)(u, \frac{v}{b_{2}})\frac{1}{\sqrt{2b_{2}\j\pi}}e^{\j\frac{d_{2}}{2b_{2}}v^{2}},
\end{eqnarray*}

where
\begin{eqnarray*}
\mathcal{F}_{r}^{\i}(u,t)&&:=\int_{\mathbb{R}}f(s,t)e^{\i us}ds.\\
\mathcal{F}_{r}^{\j}(u,v)&&:=\int_{\mathbb{R}}f(u,t)e^{\j vt}dt.\\
\mathcal{L}_{R}^{\i}(f)(u,t)&&:=\int_{\mathbb{R}}f(s,t)K_{A_{1}}^{\i}(s,u)ds.\\
\end{eqnarray*}

\end{lemma}

Using similar argument as the proof of Theorem $\ref{T53}$, we have the following Lemma.
\begin{lemma}\label{l82}
Suppose $f$ and $\mathcal{F}\in L^{1}(\mathbb{R}, \mathbb{H}) $
then
$$
f(x)= \frac{1}{2\pi}\int_{\mathbb{R}}\mathcal{F}(w)e^{\mu xw}dw,  $$
for $x$ almost everywhere,
where $\mathcal{F}(w)=\int_{\mathbb{R}}f(x)e^{-\mu xw}dx,$ $\mu$ is a  pure unit quaternion, which has unit magnitude having no real part.
\end{lemma}

Then we can prove our desired results.
\begin{theorem}\label{h823}($\textbf{ Inversion Theorem for right-sided QLCT}$)\\
Suppose $f$ and $\mathcal{L}_{R}^{\i,\j}(f) \in L^{1}(\mathbb{R}^{2}, \mathbb{H}),$
then the inversion formula of right-sided QLCT of $f $ is
\begin{eqnarray}\label{TT57}
f(s,t)= \int_{\mathbb{R}^{2}}\mathcal{L}_{R}^{\i,\j}(f)(u,v)K_{A_{2}^{-1}}^{\j}(v,t)K_{A_{1}^{-1}}^{\i} (u,s)dudv,
\end{eqnarray}
for almost everywhere $ (s,t)$.
\end{theorem}
\begin{proof}
On  one hand, by assumption $f \in L^{1}(\mathbb{R}^{2}, \mathbb{H})$, then $ \mathcal{L}_{R}^{\i}(f) \in L^{1}(\mathbb{R}, \mathbb{H})$ of variable $t$, this implies that
 $$\mathcal{L}_{R}^{\i}(u,\cdot)\frac{1}{\sqrt{2\pi b_{2}\j}}e^{\j\frac{a_{2}}{2b_{2}}(\cdot)^{2}} \in L^{1}(\mathbb{R}, \mathbb{H}).$$
 On the other hand, since $ \mathcal{L}_{R}^{\i,\j}(f)(u,v)\in L^{1}(\mathbb{R}^{2}, \mathbb{H}), $
 by the Lemma $\ref{l81}$, we have
 $$\mathcal{F}_{r}^{\j}\left(\mathcal{L}_{R}^{\i}(f)
(u, \cdot)e^{\j\frac{a_{2}}{2b_{2}}(\cdot)^{2}}\right )(u, \frac{v}{b_{2}})\in L^{1}(\mathbb{R}^{2}, \mathbb{H}), $$
 combining with  Lemma $\ref{l82}$, it  follows that
\begin{eqnarray*}
 \mathcal{L}_{R}^{\i}(f)(u, t)e^{\j\frac{a_{2}}{2b_{2}}t^{2}}&=&\int_{\mathbb{R}}
 \mathcal{F}_{r}^{\j}\bigg(\mathcal{L}_{R}^{\i}(f)
(u, t)e^{\j\frac{a_{2}}{2b_{2}}t^{2}}\bigg)(u, \frac{v}{b_{2}})\frac{1}{2b_{2}\pi}
e^{\j(\frac{1}{b_{2}}vt)}dv\\
\mathcal{L}_{R}^{\i}(f)(u, t)e^{\j\frac{a_{2}}{2b_{2}}t^{2}}&=&\int_{\mathbb{R}}
  \mathcal{L}_{R}^{\i,\j}(f)(u,v) \sqrt{2b_{2}\j\pi}e^{\j(-\frac{d_{2}}{2b_{2}}v^{2})}\frac{1}{2b_{2}\pi}
e^{\j(\frac{1}{b_{2}}vt)}dv\\
\mathcal{L}_{R}^{\i}(f)(u, t)&=&\int_{\mathbb{R}} \mathcal{L}_{R}^{\i,\j}(f)(u,v)
 \frac{1}{\sqrt{-2b_{2}\j\pi}}e^{\j(-\frac{d_{2}}{2b_{2}}v^{2}+\frac{1}{b_{2}}vt-\frac{a_{2}}{2b_{2}}t^{2})}dv\\
 &=&\int_{\mathbb{R}} \mathcal{L}_{R}^{\i,\j}(f)(u,v)K_{A_{2}^{-1}}^{\j}(v,t)dv.
\end{eqnarray*}
for almost everywhere $ t.$\\

Similarly, $f\in L^{1}(\mathbb{R}^{2}, \mathbb{H})$ and $\mathcal{L}_{R}^{\i}(f)(\cdot, t) \in L^{1}(\mathbb{R}, \mathbb{H})$  because of
\begin{eqnarray*}
\int_{\mathbb{R}}\left |\mathcal{L}_{R}^{\i}(f)(u, t)\right |du \leq \int_{\mathbb{R}^{2}}\left | \mathcal{L}_{R}^{\i,\j}(f)(u,v)\right |
dvdu,
\end{eqnarray*}
then
\begin{eqnarray*}
f(s,t)&&=\int_{\mathbb{R}} \mathcal{L}_{R}^{\i}(f)(u,t)
 \frac{1}{\sqrt{-2b_{1}\i\pi}}e^{-\i(\frac{d_{1}}{2b_{1}}u^{2}+\frac{1}{2b_{1}}us-\frac{a_{1}}{2b_{1}}s^{2})}du\\
 &&=\int_{\mathbb{R}^{2}} \mathcal{L}_{R}^{\i,\j}(f)(u,v)K_{A_{2}^{-1}}^{\j}(v,t)K_{A_{1}^{-1}}^{\i}(u,s)dvdu,
\end{eqnarray*}
for almost everywhere $ (s,t).$
\end{proof}

The similarly result can be obtained for left-sided QLCT.
\begin{theorem}\label{h824}($\textbf{ Inversion Theorem for left-sided QLCT}$)\\
Suppose $f$ and $\mathcal{L}_{L}^{\i,\j}(f) \in L^{1}(\mathbb{R}^{2}, \mathbb{H}),$
then the inversion formula of left-sided QLCT of $f $ is
\begin{eqnarray}\label{TT58}
f(s,t)= \int_{\mathbb{R}^{2}}K_{A_{1}^{-1}}^{\i} (u,s)K_{A_{2}^{-1}}^{\j}(v,t)\mathcal{L}_{L}^{\i,\j}(f)(u,v)dudv,
\end{eqnarray}
for almost everywhere $ (s,t)$.
\end{theorem}

\begin{remark}\label{rem319}
\begin{enumerate}
  \item
   In above Lemmas  and Theorems,  we can replace the imaginary units $\i$ and $\j$ by the pure unit quaternion $ \mu_{1}$ and
 $ \mu_{2}$ respectively, which are defined in Equations
$(\ref{hu37})$.
  \item
  When  $A_{1}=A_{2}=\left(
             \begin{array}{cc}
               0 &1 \\
               -1  &0  \\
             \end{array}
           \right)$,
the QLCTs   reduce to the QFTs, then the inversion theorems of right-sided and left-sided QLCTs, such as  Theorems $\ref{c419},\ref{c420},\ref{c421}$, become the inversion theorems for right-sided and left-sided QFTs.
  \item
  When $ A_{1}=\left(
             \begin{array}{cc}
               \cos \alpha & \sin \alpha \\
               -\sin \alpha  & \cos\alpha  \\
             \end{array}
           \right)
, A_{2}=\left(
             \begin{array}{cc}
               \cos \beta & \sin \beta \\
               -\sin \beta & \cos\beta  \\
             \end{array}
           \right)$,
the inversion theorems of different types of QLCTs become the inversion theorems of different types of QFRFTs.
\end{enumerate}

\end{remark}

\setcounter{equation}{0}

\section{Conclusion}\label{sec4}
This paper studied the conditions on the inversion Theorems of QFTs and QLCTs, which are summarized in the following Tables 1 and 2.
\begin{table}[h]\label{tab1}
\caption{Inversion conditions of QFTs}
\centering
\begin{tabular}{|c|c|c|}
\hline
QFT & Inversion conditions for  $f \in L^{1}(\mathbb{R}^{2}, \mathbb{H})$  & Inversion formula\\
\hline
$\multirow{10}{*}{ All  QFTs}  $& In the cross-neighborhood  of $ (x_{0},y_{0}),$ &
$ \multirow{2}{*}{  Formulas (\ref{h81}), (\ref{r1}), (\ref{l1}).}$\\
 & $ f$ is a QBVF  and belongs to $\bf{LC}$. &  \\

 \cline{2-3}

& $\mathcal{F}_{T} \in L^{1}(\mathbb{R}^{2}, \mathbb{H})$,   & $\multirow{10}{*}{ Formulas (\ref{I3}),(\ref{RR1}),(\ref{LL1}).}$ \\
 \cline{2-2}
&$\mathcal{F}_{T,n} \in L^{1}(\mathbb{R}^2,\mathbb{H}),$ & \\
&where $\mathcal{F}_{T, n} $ is the QFT of $f_{n}, n=0,1,2,3$,   &\\
&which are the components of the $f$. &\\
 \cline{2-2}
&$F_{n} \in L^{1}(\mathbb{R}^2,\mathbb{H}),$ &\\
&where $F_{n} $ is the 2D FT of $f_{n}, n=0,1,2,3$,    &\\
&which are the components of the $f$. &\\
 \cline{2-2}
&$f$ is continuous at $(0,0)$, $ F_{n}\geq 0 ,$    &\\
\hline
Two-sided QFT & $f $ has derivatives  & Formula $ (\ref{I3}).$\\
&in the $ L^{1}(\mathbb{R}^{2}, \mathbb{H})$ norm of all orders $ \leq 3.$ & \\
\hline
\end{tabular}
\end{table}

\begin{table}[h!]\label{tab2}
\caption{Inversion conditions of QLCTs}
\centering
\begin{tabular}{|c|c|c|}
\hline
 QLCTs & Inversion conditions for $f \in L^{1}(\mathbb{R}^{2}, \mathbb{H})$ & Inversion formula\\
\hline

\multirow{4}{*}{Two-sided QLCT}
  & In the cross-neighborhood  of $ (x_{0},y_{0}),$ $ f $ is a QBVF  &

 \multirow{2}{*}{ Formula $(\ref{h82})$}  \\

 & and $e^{\i\frac{a_{1}}{2b_{1}}s^{2}}f(s,t)e^{\j\frac{a_{2}}{2b_{2}}t^{2}}$ belongs to $\bf{LC}$. & \\
 \cline{2-3}
 & $\mathcal{L}_{T}^{\i,\j}(f) \in L^{1}(\mathbb{R}^{2}, \mathbb{H})$ & \multirow{3}{*}{ Formula (\ref{F3})}\\
\cline{2-2}
&$e^{\i\frac{a_{1}}{2b_{1}}s^{2}}f(s,t)e^{\j\frac{a_{2}}{2b_{2}}t^{2}}  \in L^{1}(\mathbb{R}^{2}, \mathbb{H}) $ has derivatives &
\\
& in the $ L^{1}(\mathbb{R}^{2}, \mathbb{H})$ norm of all orders $ \leq 3.$ &\\
\cline{1-3}

  \multirow{8}{*}{ Right-sided QLCT}  &   In the cross-neighborhood  of $ (x_{0},y_{0}),$ &

  \multirow{3}{*}{ Formula $(\ref{xx2})$}  \\
   & $ f $ is a QBVF,  $e^{\i\frac{a_{1}}{2b_{1}}s^{2}}f_{a}(s,t)e^{\j\frac{a_{2}}{2b_{2}}t^{2}}$ and  &  \\
   & $ e^{-\i\frac{a_{1}}{2b_{1}}s^{2}}f_{b}(s,t)e^{\j\frac{a_{2}}{2b_{2}}t^{2}}$ belongs to $\bf{LC}$. & \\
   \cline{2-3}

  & $\mathcal{L}_{R}^{\i,\j}(f) \in L^{1}(\mathbb{R}^{2}, \mathbb{H})$ &  Formula (\ref{TT57})\\

   \cline{2-3}
 &$  e^{\i\frac{a_{1}}{2b_{1}}s^{2}}f_{a}(s,t)e^{\j\frac{a_{2}}{2b_{2}}t^{2}}$ and $ e^{-\i\frac{a_{1}}{2b_{1}}s^{2}}f_{b}(s,t)e^{\j\frac{a_{2}}{2b_{2}}t^{2}}  $  &    \multirow{4}{*}{ Formula  (\ref{rr1})}
 \\ & both have derivatives  in the $ L^{1}(\mathbb{R}^{2}, \mathbb{H})$ norm of all orders $ \leq 3.$ &
\\  \cline{2-2}
 & $ \mathcal{L}_{T}^{\i,\j}(f_{a})$ and $\mathcal{L}_{T}^{-\i,\j}(f_{b})  \in L^{1}(\mathbb{R}^{2}, \mathbb{H})$, & \\
 \cline{1-3}

 \multirow{8}{*}{ Left-sided  QLCT}  &   In the cross-neighborhood  of $ (x_{0},y_{0}),$ &

  \multirow{3}{*}{ Formula $(\ref{xxx2})$}  \\
   & $ f $ is a QBVF,  $e^{\i\frac{a_{1}}{2b_{1}}s^{2}}f_{d}(s,t)e^{\j\frac{a_{2}}{2b_{2}}t^{2}}$ and  &  \\
   & $ e^{\i\frac{a_{1}}{2b_{1}}s^{2}}f_{e}(s,t)e^{-\j\frac{a_{2}}{2b_{2}}t^{2}}$ belongs to $\bf{LC}$. & \\
   \cline{2-3}

  & $\mathcal{L}_{L}^{\i,\j}(f) \in L^{1}(\mathbb{R}^{2}, \mathbb{H})$ &  Formula (\ref{TT58})\\

   \cline{2-3}
 &$  e^{\i\frac{a_{1}}{2b_{1}}s^{2}}f_{d}(s,t)e^{\j\frac{a_{2}}{2b_{2}}t^{2}}$ and $ e^{\i\frac{a_{1}}{2b_{1}}s^{2}}f_{e}(s,t)e^{-\j\frac{a_{2}}{2b_{2}}t^{2}} $&    \multirow{4}{*}{ Formula  (\ref{ll1})}
 \\ & both have derivatives  in the $ L^{1}(\mathbb{R}^{2}, \mathbb{H})$ norm of all orders $ \leq 3.$ &
\\  \cline{2-2}
 & $ \mathcal{L}_{T}^{\i,\j}(f_{d})$ and $\mathcal{L}_{T}^{\i,-\j}(f_{e})  \in L^{1}(\mathbb{R}^{2}, \mathbb{H})$. & \\
 \cline{1-3}
  &\multicolumn{2}{c| }{ Where  $f_{a}, f_{b}, f_{e}, f_{d}$  is defined by Equation (\ref{f8})}   \\
  \hline
\end{tabular}
\end{table}

Further investigations on this topic will be focus on the applications of QLCT to problems of color image processing.

\section{Acknowledgements}

The authors acknowledge financial support from the National Natural Science Funds for Young Scholars (No. 11401606, 11501015) and University of Macau No. MYRG2015-00058-FST, MYRG099(Y1-L2)-FST13-KKI and the Macao Science and Technology Development Fund FDCT/094/2011A, FDCT/099/2012/A3.


\end{document}